\documentclass[11pt]{article}
\usepackage{amsmath,amsthm}
\usepackage{amssymb}
\usepackage{enumerate}
\usepackage{graphicx,tikz}
\usepackage[mathscr]{eucal}
\usepackage[cm]{fullpage}
\usepackage{xcolor}
\usepackage[english]{babel}
\usepackage[latin1]{inputenc}

\usepackage[normalem]{ulem}
\newcommand{\End}{\mathrm{End}}
\newcommand{\wEnd}{\mathrm{wEnd}}
\newcommand{\sEnd}{\mathrm{sEnd}}
\newcommand{\swEnd}{\mathrm{swEnd}}
\newcommand{\Aut}{\mathrm{Aut}}
\theoremstyle{plain}
\newtheorem{theorem}{Theorem}[section]
\newtheorem{proposition}[theorem]{Proposition}
\newtheorem{lemma}[theorem]{Lemma}
\newtheorem{corollary}[theorem]{Corollary}

\newtheorem{algorithm}[theorem]{Algorithm} 

\def\im{\mathop{\mathrm{Im}}\nolimits} 
\def\ker{\mathop{\mathrm{Ker}}\nolimits} 
\def\rank{\mathop{\mathrm{rank}}\nolimits} 

\numberwithin{equation}{section}

\begin{document}

\title{On monoids of endomorphisms of a cycle graph}

\author{I. Dimitrova, V.H. Fernandes\footnote{This work is funded by national funds through the FCT - Funda\c c\~ao para a Ci\^encia e a Tecnologia, I.P., under the scope of the projects UIDB/00297/2020 and UIDP/00297/2020 (NovaMath - Center for Mathematics and Applications).},  J. Koppitz and T.M. Quinteiro\footnote{This work is funded by national funds through the FCT - Funda\c c\~ao para a Ci\^encia e a Tecnologia, I.P., under the scope of the projects UIDB/00297/2020 and UIDP/00297/2020 (NovaMath - Center for Mathematics and Applications).}
}


\maketitle

\renewcommand{\thefootnote}{}

\footnote{2020 \emph{Mathematics Subject Classification}: 05C38, 20M10, 20M20, 05C25}

\footnote{\emph{Key words}: Graph endomorphisms, Cycle graphs, Generators, Rank}

\renewcommand{\thefootnote}{\arabic{footnote}}
\setcounter{footnote}{0}

\vspace*{-4em}
\begin{abstract}
In this paper we consider endomorphisms of an undirected cycle graph from Semigroup Theory perspective. 
Our main aim is to present a process to determine sets of generators with minimal cardinality 
for the monoids $\wEnd C_n$ and $\End C_n$ of all weak endomorphisms and all endomorphisms of an undirected cycle graph $C_n$ with $n$ vertices. 
We also describe Green's relations and regularity of these monoids and calculate their cardinalities. 
\end{abstract}

\section*{Introduction}

It is well known that sets of all endomorphisms of a certain type of a graph constitute submonoids of the monoid of transformations on its vertices. 
They are generalizations of the automorphism group of the graph. The aim of the research in this line is to establish the relationship between graph theory and semigroup theory. Monoids of endomorphisms of graphs have valuable applications, many of which are related to automata theory (see \cite{Kelarev:2003}).  
In recent years, many authors have paid attention to endomorphism monoids of graphs and a large number of interesting results concerning graphs and algebraic properties of their endomorphism monoids have been obtained (see, for example, \cite{Bottcher&Knauer:1992,Fan:1996,Fan:1997,Gu&Hou:2016,Hou&Gu:2016,Hou&Gu&Shang:2014,Hou&Luo&Fan:2012,Hou&Song&Gu:2017,Kelarev&Praeger:2003,Knauer&Wanichsombat:2014, Li:2003,Wilkeit:1996}).
In 1987, Knauer and Wilkeit questioned for which graphs $G$, the endomorphism monoid of $G$ is regular \cite{Marki:1988}. 
After this question was posed, many special graphs and their endomorphism monoids were studied. 
For instance, the connected bipartite graphs with regular endomorphism monoid were completely described by Wilkeit in 1996 \cite{Wilkeit:1996}. 
Although many results along these lines have been obtained (see some of the papers mentioned above), 
a characterization of all graphs with a regular endomorphism monoid seems to be very complicated. 
However, for weak endomorphisms, Wilkeit, still in the paper \cite{Wilkeit:1996}, also characterized all graphs with a regular weak endomorphism monoid. 

\smallskip 

Let $M$ be a monoid. The rank of $M$  
is defined to be the minimum size of a generating set of $M$, i.e. 
the minimum of the set $\{|X|\mid \mbox{$X\subseteq M$ and $X$ generates $M$}\}$. 

Over the last forty years, ranks of various well known semigroups have been calculated (see, for example, 
\cite{Araujo&al:2015,Fernandes:2000,Fernandes:2001,Fernandes&Sanwong:2014,Gomes&Howie:1987,Gomes&Howie:1992}). 
In 2020, the authors characterized the endomorphisms and weak endomorphisms of a finite undirected path $P_n$ with $n$ vertices; the ranks of its endomorphism monoids and weak endomorphism monoids were determined in \cite{DFKQ:2020}. Later, in \cite{DFKQ:2021}, they determined the ranks of the monoids of all injective partial endomorphisms and of all partial automorphisms of $P_n$, described their Green's relations and calculated the cardinalities of these two monoids.

\smallskip 

The main aim of this paper is to build a process to determine the ranks of the monoids of all weak endomorphisms and all endomorphisms of an undirected cycle graph $C_n$ with $n$ vertices. We also describe Green's relations and the regularity of these monoids and calculate their cardinalities. Monoids of strong endomorphisms and of strong weak endomorphisms of $C_n$ are also considered in this paper.

\smallskip 

Let $G=(V,E)$ be a simple graph (i.e. undirected, without loops or multiple edges). 
Let $\alpha$ be a full transformation of $V$. We say that $\alpha$ is: 
\begin{itemize}
\item[--] an \textit{endomorphism} of $G$ if $\{u,v\}\in E$ implies  $\{u\alpha,v\alpha\}\in E$, for all $u,v\in V$;
\item[--] a \textit{weak endomorphism} of $G$ if $\{u,v\}\in E$ and $u\alpha\ne v\alpha$ imply  $\{u\alpha,v\alpha\}\in E$, for all $u,v\in V$;
\item[--] a \textit{strong endomorphism} of $G$ if $\{u,v\}\in E$ if and only if  $\{u\alpha,v\alpha\}\in E$, for all $u,v\in V$;
\item[--] a \textit{strong weak endomorphism} of $G$ if $\{u,v\}\in E$ and $u\alpha\ne v\alpha$ if and only if $\{u\alpha,v\alpha\}\in E$, for all $u,v\in V$;
\item[--] an \textit{automorphism} of $G$ if $\alpha$ is a bijective strong endomorphism (i.e. $\alpha$ is bijective and $\alpha$ and $\alpha^{-1}$ are both endomorphisms). 
For finite graphs, any bijective week endomorphism is an automorphism. 
\end{itemize}

Denote by:
\begin{itemize}
\item[--] $\End(G)$ the set of all endomorphisms of $G$;
\item[--] $\wEnd(G)$ the set of all weak endomorphisms of $G$;
\item[--] $\sEnd(G)$ the set of all strong endomorphisms of $G$;
\item[--] $\swEnd(G)$ the set of all strong weak endomorphisms of $G$;
\item[--] $\Aut(G)$ the set of all automorphisms of $G$. 
\end{itemize}

Clearly, $\End(G)$, $\wEnd(G)$, $\sEnd(G)$, $\swEnd(G)$ and $\Aut(G)$ are monoids under composition of maps. Moreover, $\Aut(G)$ is also a group.  
It is also clear that $\Aut(G)\subseteq\sEnd(G)\subseteq\End(G)\subseteq\wEnd(G)$ and $\Aut(G)\subseteq\sEnd(G)\subseteq\swEnd(G)\subseteq\wEnd(G)$
\begin{center}
\begin{tikzpicture}[scale=0.5]
\draw (0,0) node{$\bullet$} (0,1) node{$\bullet$} (-1,2) node{$\bullet$} (1,2) node{$\bullet$} (0,3) node{$\bullet$}; 
\draw (1.5,0.05) node{\small$\Aut(G)$} (1.7,1.05) node{\small$\sEnd(G)$} (-2.8,2.05) node{\small$\swEnd(G)$} (2.5,2.05) node{\small$\End(G)$} (1.7,3.15) node{\small$\wEnd(G)$};
\draw[thick] (0,0) -- (0,1); 
\draw[thick] (0,1) -- (-1,2); 
\draw[thick] (0,1) -- (1,2); 
\draw[thick] (-1,2) -- (0,3);
\draw[thick] (1,2) -- (0,3);
\end{tikzpicture}
\end{center}
(these inclusions may not be strict).

\smallskip 

For any positive integer $n$, let $\Omega_n=\{1,2,\ldots,n\}$ and 
denote by $\mathscr{T}_n$ the monoid of all transformations of $\Omega_n$ under composition of maps.
Also, denote by $\mathscr{S}_n$ the symmetric group on $\Omega_n$, i.e. the subgroup of $\mathscr{T}_n$ of all permutations of $\Omega_n$, 
and by $\mathscr{D}_{2n}$ a dihedral group of order $2n$, i.e. 
$\mathscr{D}_{2n}=\langle g,h\mid g^n=h^2=1, gh=hg^{n-1}\rangle$. 
Recall that, for $n\geqslant3$, we can consider $\mathscr{D}_{2n}$ as a subgroup of $\mathscr{S}_n$. In fact, for instance, taking 
$$
g=\begin{pmatrix} 
1&2&\cdots&n-1&n\\
2&3&\cdots&n&1
\end{pmatrix}, 
h=\begin{pmatrix} 
1&2&\cdots&n-1&n\\
n&n-1&\cdots&2&1
\end{pmatrix}\in \mathscr{S}_n, 
$$ 
we have 
$
\mathscr{D}_{2n}=\langle g,h\rangle=\{1,g,g^2,\ldots,g^{n-1}, h,hg,hg^2,\ldots,hg^{n-1}\}. 
$ 
In particular,  $\mathscr{D}_{2n}$ is a rank two semigroup, monoid and group. 

Observe that
$$
g^k=\begin{pmatrix} 
1&2&\cdots&n-k&n-k+1&\cdots&n\\
1+k&2+k&\cdots&n&1&\cdots&k
\end{pmatrix}, 
\quad\text{i.e.}\quad  
ig^k=\left\{\begin{array}{ll}
i+k & \mbox{if $1\leqslant i\leqslant n-k$}\\
i+k-n & \mbox{if $n-k+1\leqslant i\leqslant n$,}  
\end{array}\right.
$$
and 
$$
hg^k=\begin{pmatrix} 
1&\cdots&k&k+1&\cdots&n\\
k&\cdots&1&n&\cdots&k+1
\end{pmatrix}, 
\quad\text{i.e.}\quad  
ihg^k=\left\{\begin{array}{ll}
k-i+1 & \mbox{if $1\leqslant i\leqslant k$}\\
n+k-i+1 & \mbox{if $k+1\leqslant i\leqslant n$,} 
\end{array}\right.
$$
for $0\leqslant k\leqslant n-1$. 

\smallskip

Let $n\geqslant3$ be a positive integer. 
Denote by $C_n$ a \textit{cycle graph} with $n$ vertices and fix 
$$
C_n=(\Omega_n,\{\{i,i+1\}\mid i=1,2,\ldots,n-1\}\cup\{\{1,n\}\}).
$$ 
\begin{center}
\begin{tikzpicture}
\draw (0,0) node{$\bullet$} (0,2) node{$\bullet$} (-1,1) node{$\bullet$} (1,1) node{$\bullet$}; 
\draw (0.7,0.3) node{$\bullet$} (0.7,1.7) node{$\bullet$} (-0.7,1.7) node{$\bullet$}; 
\draw (0,-0.2) node{$\scriptstyle5$} (0,2.25) node{$\scriptstyle1$} (-1.4,1) node{$\scriptstyle n-1$} (1.18,1) node{$\scriptstyle3$}; 
\draw (0.9,0.3) node{$\scriptstyle4$} (0.95,1.7) node{$\scriptstyle2$} (-0.95,1.7) node{$\scriptstyle n$}; 
\draw[thick] (0,0) arc (-90:180:1);
\draw[thin,dashed] (0,0) arc (-90:-180:1);
\end{tikzpicture}
\end{center}

For convenience, on several occasions throughout the text, we will take addition modulo $n$ with $\{1,2,\dots,n\}$ as set of representatives. 
For instance, when $n$ is considered as vertex of $C_n$, the expression $n+1$ also denotes the vertex $1$ of $C_n$ and so, 
in this line, we can write $C_n=(\Omega_n,\{\{i,i+1\}\mid i=1,2,\ldots,n\})$.

\smallskip 

As usual, for $1\leqslant i\leqslant j\leqslant n$, we denote by $[i,j]$ the interval $\{i,i+1,\ldots,j\}$ of $\Omega_n$.  
To an interval of $\Omega_n$ or to an union of intervals of $\Omega_n$ of the form 
$[j,n]\cup[1,i]$, for some $1\leqslant i< j-1\leqslant n-1$, we call an \textit{arc} of $\Omega_n$. 

\smallskip

Throughout this paper we always consider $n\geqslant3$.

\smallskip

This paper is organized as follows: in Section \ref{basics}, we give descriptions of the various types of endomorphisms of $C_n$, 
\textit{completely} characterize $\Aut(C_n)$, $\sEnd(C_n)$ and $\swEnd(C_n)$, determine the cardinalities of $\End(C_n)$ and $\wEnd(C_n)$ and state some basic auxiliary results; Section \ref{regularity} is dedicated to characterize the regular elements of $\End(C_n)$ and $\wEnd(C_n)$, 
as well as to describe Green's relations $\mathcal{R}$, $\mathcal{L}$ and $\mathcal{D}$ on these two monoids; finally, in Section \ref{ranksec}, we exhibit a process to determine sets of generators with minimal cardinality for the monoids $\wEnd C_n$ and $\End C_n$. 

\smallskip 

For general background on Semigroup Theory and standard notations, we refer to Howie's book \cite{Howie:1995}. 
Regarding Algebraic Graph Theory, we refer to Knauer's book \cite{Knauer:2011}.

\smallskip 

We would like to point out that we made massive use of computational tools, namely GAP \cite{GAP4}.

\section{Some basic properties}\label{basics} 

We start this section by observing that, for $1\leqslant i,j\leqslant n$, we have that $\{i,j\}$ is an edge of $C_n$ if and only if $|i-j|\in\{1,n-1\}$. 
In particular, if $1\leqslant i\leqslant j\leqslant n$ then $\{i,j\}$ is an edge of $C_n$ if and only if $j-i\in\{1,n-1\}$. 
It is also a routine matter to show the following properties: 

\begin{proposition}\label{chars}
Let $\alpha\in \mathscr{T}_n$. Then: 
\begin{enumerate}
\item $\alpha\in\wEnd(C_n)$ if and only if $|(i+1)\alpha-i\alpha|\in \{0,1,n-1\}$ for all $1\leqslant i\leqslant n$; 
\item $\alpha\in\End(C_n)$ if and only if $|(i+1)\alpha-i\alpha|\in \{1,n-1\}$ for all $1\leqslant i\leqslant n$; 
\item $\alpha\in\swEnd(C_n)$ if and only if 
\begin{enumerate}
\item $|(i+1)\alpha-i\alpha|\in \{0,1,n-1\}$ for all $1\leqslant i\leqslant n$, and 
\item $|j\alpha-i\alpha|\in\{1,n-1\}$ implies $j-i\in\{1,n-1\}$ for all $1\leqslant i\leqslant j\leqslant n$; 
\end{enumerate} 
\item $\alpha\in\sEnd(C_n)$ if and only if 
\begin{enumerate}
\item $|(i+1)\alpha-i\alpha|\in \{1,n-1\}$ for all $1\leqslant i\leqslant n$, and 
\item $|j\alpha-i\alpha|\in\{1,n-1\}$ implies $j-i\in\{1,n-1\}$ for all $1\leqslant i\leqslant j\leqslant n$. 
\end{enumerate} 
\end{enumerate}
\end{proposition}

Notice that $C_3$ is also the complete graph on $3$ vertices. 
Therefore, it is clear that $\End(C_3)=\sEnd(C_3)=\Aut(C_3)=\mathscr{S}_3=\mathscr{D}_{2\times3}$ and $\swEnd(C_3)=\wEnd(C_3)=\mathscr{T}_3$. 

\begin{proposition}\label{aut}
For $n\geqslant3$, $\Aut(C_n) = \mathscr{D}_{2n}$. 
\end{proposition}
\begin{proof}
We begin by observing that, clearly, $g,h\in\Aut(C_n)$. Hence $\mathscr{D}_{2n}\subseteq\Aut(C_n)$. 

Conversely, let us first take $\alpha\in\Aut(C_n)$ such that $n\alpha=n$. 

Since $\{1,n\}$ is an edge of $C_n$ then $\{1\alpha,n\}=\{1\alpha,n\alpha\}$ is an edge of $C_n$ and so $1\alpha=1$ or $1\alpha=n-1$. 

If $1\alpha=1$ then $\alpha=1$. In fact, $\{1,2\alpha\}=\{1\alpha,2\alpha\}$ is an edge of $C_n$ and so, as $n\alpha=n$, we must have $2\alpha=2$. 
Now, if $\alpha$ were not the identity then we could take the smallest $2<i<n$ such that $i\alpha\neq i$. In this case, as $\{i-1,i\alpha\}=\{(i-1)\alpha,i\alpha\}$ 
 is an edge of $C_n$ and $i\alpha\neq i$, we could only have $i\alpha=i-2=(i-2)\alpha$, which is a contradiction.  

 Suppose that $1\alpha=n-1$. Then $\{n-1,2\alpha\}=\{1\alpha,2\alpha\}$ is an edge of $C_n$ and so, as $n\alpha=n$, we must have $2\alpha=n-2$. 
 Admit there exists $2<j<n$ such that $j\alpha\neq n-j$ and take the smallest one. 
 Then $\{n-j+1,j\alpha\}=\{(j-1)\alpha,j\alpha\}$ is an edge of $C_n$ and so $j\alpha=n-j+2=(j-2)\alpha$, which is a contradiction.  
 Hence $j\alpha=n-j$, for $1\leqslant j<n$, i.e. $\alpha=hg^{n-1}$. 
 
Now, let $\alpha$ be an arbitrary element of $\Aut(C_n)$. Then $\alpha g^{n-n\alpha}\in\Aut(C_n)$ and $n\alpha g^{n-n\alpha}=n$. 
Hence, in view of what we proved above, $\alpha g^{n-n\alpha}=1$ or $\alpha g^{n-n\alpha}=hg^{n-1}$ and so $\alpha=g^{n\alpha}$ or $\alpha=hg^{n\alpha-1}$. 
Thus, in both cases, we obtain $\alpha\in\mathscr{D}_{2n}$, as required.  
\end{proof}

For $n\geqslant3$, it follows that $|\Aut(C_n)| = 2n$.  

\begin{proposition}\label{send}
For $n=3$ and $n\geqslant5$, $\sEnd(C_n) = \Aut(C_n)$. 
\end{proposition}
\begin{proof} 
We will use Proposition \ref{chars} several times without further mention. 

Let $\alpha\in\sEnd(C_n)$. Suppose, by contradiction, there exist $1\leqslant i<j\leqslant n$ such $i\alpha=j\alpha$. 
Then 
$
|(i+1)\alpha-j\alpha|=|(i+1)\alpha-i\alpha|\in\{1,n-1\}
$ 
(since in particular $\alpha\in\End(C_n)$) and so, as $i+1\leqslant j$, we have $j-(i+1)\in\{1,n-1\}$. 
Since $j<i+n$, it follows that $j=i+2$. 

If $i=1$ and $i+2=n$ we would have $|1\alpha-n\alpha|=|i\alpha-j\alpha|=0$, which is a contradiction. 
Therefore $i>1$ or $i+2<n$. 
If $i>1$ then 
$
|(i+2)\alpha-(i-1)\alpha|=|j\alpha-(i-1)\alpha|=|i\alpha-(i-1)\alpha|\in\{1,n-1\}
$
and so $3=(i+2)-(i-1)\in\{1,n-1\}$, which is a contradiction since $n\neq4$. 
If $i+2<n$ then 
$
|(i+3)\alpha-i\alpha|=|(i+3)\alpha-j\alpha|=|(i+3)\alpha-(i+2)\alpha|\in\{1,n-1\}
$
and so $3=(i+3)-i\in\{1,n-1\}$, which again is a contradiction since $n\neq4$. 

Thus, $\alpha$ is injective and so $\alpha\in\Aut(C_n)$, as required. 
\end{proof}

It follows that $|\sEnd(C_n)| = 2n$, for $n=3$ and $n\geqslant5$. 

For $n=4$, by inspection of the elements of $\mathscr{T}_4$, we can show that $|\sEnd(C_4)|=32$, whence $\sEnd(C_4)$ has $24$ non-automorphism elements, 
including $16$ with rank $3$ and $8$ with rank $2$. Moreover, it is a routine matter to verify (for instance, using GAP \cite{GAP4}) that $\{g,h,u\}$, with $u$ any one of the $16$ elements of $\sEnd(C_4)$ with rank $3$,  forms a minimum size generating set of $\sEnd(C_4)$, whence $\sEnd(C_4)$ has rank three. 
Furthermore, we can also verify that $\End(C_4)=\sEnd(C_4)$. 

This last equality is also valid for all odd positive integer $n$ greater than or equal to three, as we will show in Proposition \ref{endodd}. 
Before that, we will prove some lemmas. 

\begin{lemma} \label{ims}
Let $1\leqslant i\leqslant j\leqslant n$ and $\alpha\in\wEnd(C_n)$. Then, there exist $1\leqslant a\leqslant b\leqslant n$ such that 
$[i,j]\alpha=[a,b]$ or $[i,j]\alpha=[b,n]\cup[1,a]$. 
\end{lemma}
\begin{proof}
Let $k=j-i$. We will proceed by induction on $k$.

If $k=0$ then $[i,j]\alpha=\{i\alpha\}=[i\alpha,i\alpha]$, which satisfies the statement of the lemma.

Now, suppose that the lemma is valid for some $k\geqslant0$ and consider the interval $[i,i+k+1]$. 
Then, there exist $1\leqslant a\leqslant b\leqslant n$ such that 
$[i,i+k]\alpha=[a,b]$ (Case 1) or $[i,i+k]\alpha=[b,n]\cup[1,a]$ (Case 2). 
Observe that $[i,i+k+1]\alpha=[i,i+k]\alpha\cup\{(i+k+1)\alpha\}$. 

By Proposition \ref{chars}, $|(i+k+1)\alpha-(i+k)\alpha|\in\{0,1,n-1\}$, whence $(i+k+1)\alpha=(i+k)\alpha+t$, for some $t\in\{1-n,-1,0,1,n-1\}$. 
If $t=0$ then $(i+k+1)\alpha=(i+k)\alpha$ and so $[i,i+k+1]\alpha=[i,i+k]\alpha$. Thus, we can suppose that $t\neq0$.  

Let $c=(i+k)\alpha$. 

\smallskip 

First, we admit Case 1. 

Suppose that $t=1$. Then $(i+k+1)\alpha=(i+k)\alpha+1=c+1$. Since $c\in[a,b]$, we have $a\leqslant c<b$ or $c=b$. 
If $a\leqslant c<b$ then $a<c+1\leqslant b$, whence $(i+k+1)\alpha=c+1\in[a,b]$ and so $[i,i+k+1]\alpha=[a,b]\cup\{c+1\}=[a,b]$. 
On the other hand, if $c=b$ then $[i,i+k+1]\alpha=[a,b]\cup\{b+1\}=[a,b+1]$. 

Next, suppose that $t=-1$. Then $(i+k+1)\alpha=(i+k)\alpha-1=c-1$. Since $c\in[a,b]$, we have $a<c\leqslant b$ or $c=a$. 
If $a<c\leqslant b$ then $a\leqslant c-1< b$, whence $(i+k+1)\alpha=c-1\in[a,b]$ and so $[i,i+k+1]\alpha=[a,b]\cup\{c-1\}=[a,b]$. 
On the other hand, if $c=a$ then $[i,i+k+1]\alpha=[a,b]\cup\{a-1\}=[a-1,b]$. 

If $t=n-1$ then $(i+k+1)\alpha=(i+k)\alpha+n-1=c+n-1$, whence $(i+k+1)\alpha=n$ and $c=1$ and so $1\in[a,b]$, 
which implies that $a=1$ and so 
$[i,i+k+1]\alpha=[a,b]\cup\{n\}=[1,b]\cup\{n\}=[n,n]\cup[1,b]$.  

If $t=1-n$ then $(i+k+1)\alpha=(i+k)\alpha+1-n=c+1-n$, whence $(i+k+1)\alpha=1$ and $c=n$ and so $n\in[a,b]$, 
which implies that $b=n$ and so 
$[i,i+k+1]\alpha=[a,b]\cup\{1\}=[a,n]\cup\{1\}=[a,n]\cup[1,1]$.  

\smallskip 

Now, admit Case 2. 

Suppose that $t=1$. Then $(i+k+1)\alpha=(i+k)\alpha+1=c+1$, with $c\in[b,n]\cup[1,a]$.  
If $c+1\in[b,n]\cup[1,a]$ then $[i,i+k+1]\alpha=[b,n]\cup[1,a]\cup\{c+1\}=[b,n]\cup[1,a]$.  
If $c+1\not\in[b,n]\cup[1,a]$ then $c=a$ (and $a<b$) and so $[i,i+k+1]\alpha=[b,n]\cup[1,a]\cup\{a+1\}=[b,n]\cup[1,a+1]$. 

Next, suppose that $t=-1$. Then $(i+k+1)\alpha=(i+k)\alpha+1=c-1$, with $c\in[b,n]\cup[1,a]$.  
If $c-1\in[b,n]\cup[1,a]$ then $[i,i+k+1]\alpha=[b,n]\cup[1,a]\cup\{c-1\}=[b,n]\cup[1,a]$.  
If $c-1\not\in[b,n]\cup[1,a]$ then $c=b$ (and $a<b$) and so $[i,i+k+1]\alpha=[b,n]\cup[1,a]\cup\{b-1\}=[b-1,n]\cup[1,a]$.

If $t=n-1$ then $(i+k+1)\alpha=(i+k)\alpha+n-1=c+n-1$, whence $(i+k+1)\alpha=n$ and $c=1$ and so 
$[i,i+k+1]\alpha=[b,n]\cup[1,a]\cup\{n\}=[b,n]\cup[1,a]$.  

Finally, if $t=1-n$ then $(i+k+1)\alpha=(i+k)\alpha+1-n=c+1-n$, whence $(i+k+1)\alpha=1$ and $c=n$ and so 
$[i,i+k+1]\alpha=[b,n]\cup[1,a]\cup\{1\}=[b,n]\cup[1,a]$, as required. 
\end{proof}

Let $\alpha\in\wEnd(C_n)$. Then, by the previous lemma, we have $\im(\alpha)=[a,b]$ or $\im(\alpha)=[b,n]\cup[1,a]$, for some $1\leqslant  a\leqslant  b\leqslant  n$. 
Thus, we get 
\begin{equation}\label{trans}
\im(\alpha g^{n-a+1})=[1,\rank(\alpha)],~\mbox{in the first case, and}~ 
\im(\alpha g^{n-b+1})=[1,\rank(\alpha)],~\mbox{in the second case.} 
\end{equation}
Notice that $\ker(\alpha)=\ker(\alpha\sigma)$, for any permutation $\sigma$ of $\Omega_n$. 

\begin{lemma}\label{wpar} 
Let $\alpha\in\wEnd(C_n)$. Then $\alpha\in\mathscr{D}_{2n}$ or $\rank(\alpha)\leqslant \lfloor\frac{n}{2}\rfloor+1$. 
Moreover, if $n$ is an even positive integer and $\rank(\alpha)=\frac{n}{2}+1$ then $\alpha\in\End(C_n)$. 
\end{lemma}
\begin{proof}
Let $k=\rank(\alpha)$ and suppose that $\alpha\not\in\mathscr{D}_{2n}$. Then $1\leqslant k\leqslant n-1$. 

Now, by (\ref{trans}), we may take $0\leqslant t\leqslant n-1$ such that $\im(\alpha g^t)=[1,k]$. Let $\beta=\alpha g^t$. 
Let $i\in 1\beta^{-1}$ and $j\in k\beta^{-1}$. 

Suppose that $i<j$. As $1, k\in [i,j]\beta\subseteq [1,k]$ and $n\not\in\im(\beta)$, by Lemma \ref{ims}, 
we deduce that  $[i,j]\beta=[1,k]=\im(\beta)$. 
On the other hand, we must have $[2,k-1]\subseteq [j+1,n]\beta\cup[1,i-1]\beta$. 
Hence, $j-i+1\geqslant k$ and $(n-(j+1)+1)+(i-1)\geqslant k-2$, i.e. $k\leqslant  j-i+1\leqslant n-k+2$. Thus $k\leqslant \lfloor\frac{n}{2}\rfloor+1$. 

In this case, if $n$ is even and $k=\frac{n}{2}+1$ then $j-i+1=\frac{n}{2}+1$ and we must have 
$$
\beta=\left(\begin{array}{ccc|ccccc|ccc}
1 &\cdots &i-1 & i & i+1 &\cdots &j-1&j &j+1& \cdots &n \\
 j-\frac{n}{2} & \cdots & 2 & 1 & 2 & \cdots & \frac{n}{2}& \frac{n}{2}+1 & \frac{n}{2} & \cdots & j-\frac{n}{2}+1
\end{array}\right)\in\End(C_n). 
$$

If $i>j$ then a reasoning similar to the first case allows us to also conclude that $k\leqslant \lfloor\frac{n}{2}\rfloor+1$. 
In this case, for an even $n$ and $k=\frac{n}{2}+1$, we get 
$$
\beta=\left(\begin{array}{ccc|ccccc|ccc}
1 &\cdots &j-1 & j & j+1 &\cdots &i-1&i &i+1& \cdots &n \\
 n-i+2 & \cdots & \frac{n}{2} & \frac{n}{2}+1 & \frac{n}{2} & \cdots & 2& 1 & 2 & \cdots &n-i+1
\end{array}\right)\in\End(C_n). 
$$

Thus, in both cases we obtain $k\leqslant \lfloor\frac{n}{2}\rfloor+1$ and, 
if $n$ is even and $k=\frac{n}{2}+1$, we have $\alpha=\beta g^{n-t}\in\End(C_n)$, as required. 
\end{proof} 

\begin{lemma}\label{eqim}
Let $\alpha\in \End(C_n)$. Let $1\leqslant i\leqslant j\leqslant n$ be such that $i\alpha=j\alpha$. Then $j-i$ is even. 
\end{lemma}
\begin{proof}
Let $i\leqslant \ell<j$. Then, by Proposition \ref{chars},  there exists $t_\ell\in\{1-n,-1,1,n-1\}$ such that $(\ell+1)\alpha=\ell\alpha+t_\ell$. 
Then $i\alpha=j\alpha=i\alpha+\sum_{\ell=i}^{j-1}t_\ell$ and so $\sum_{\ell=i}^{j-1}t_\ell=0$. 

Define $T_t=\{i\leqslant \ell< j\mid t_\ell=t\}$, for  $t\in\{1-n,-1,1,n-1\}$. Then $|T_{1-n}|+|T_{-1}|+|T_{1}|+|T_{n-1}|=j-i$, 
whence $0\leqslant |T_t|\leqslant j-i$ for all $t\in\{1-n,-1,1,n-1\}$. 
Moreover, 
$$
0=\sum_{\ell=i}^{j-1}t_\ell=(1-n)|T_{1-n}|-|T_{-1}|+|T_{1}|+(n-1)|T_{n-1}|. 
$$

Suppose, by contradiction, that $|T_{1-n}|<|T_{n-1}|$. Then $|T_{n-1}|=|T_{1-n}|+r$, for some $r\geqslant1$, and so 
$(n-1)|T_{n-1}|+(1-n)|T_{1-n}|=(n-1)r\geqslant n-1$. 
Moreover, $|T_{n-1}|\geqslant1$, whence $|T_{-1}| <j-i$ and so $|T_{1}|-|T_{-1}|\geqslant -|T_{-1}|> i-j\geqslant 1-n$, 
which implies that 
$$
0=(|T_{1}|-|T_{-1}|)+((n-1)|T_{n-1}|+(1-n)|T_{1-n}|)>(1-n)+(n-1)=0,
$$
a contradiction. Similarly, we cannot have $|T_{1-n}|>|T_{n-1}|$. 

Hence $|T_{1-n}|=|T_{n-1}|$ and so $|T_{1}|=|T_{-1}|$, from which follows that $j-i=2|T_{1}|+2|T_{n-1}|$. Thus $j-i$ is an even number, as required. 
\end{proof}

The following result can be found in \cite{Michels&Knauer:2009}. 
Here, in line with our context, we present another proof. 

\begin{proposition}[\cite{Michels&Knauer:2009}]\label{endodd}
For all odd positive integer $n\geqslant3$, $\End(C_n) = \Aut(C_n)$. 
\end{proposition}
\begin{proof}
It suffices to show that if $\End(C_n)$ contains a non-automorphism then $n$ is even. 

Let us suppose there exists $\alpha\in\End(C_n)$ such that $\alpha\not\in\Aut(C_n)$. 
Then $i\alpha=j\alpha$, for some $1\leqslant i<j\leqslant n$. 
Let $\beta=g^{i-1}\alpha g^{n-i\alpha+1}$. Hence $\beta\in\End(C_n)$ and $1\beta=(j-i+1)\beta=1$. 
Notice that, since $1<j-i+1$, this last equality implies that $\beta\not\in\Aut(C_n)$. 

Take the largest $1<r\leqslant n$ such that $r\beta=1$. Then, by Lemma \ref{eqim}, $r-1$ is even. 
On the other hand, by Proposition \ref{chars}, we have $|1\beta-n\beta|\in\{1,n-1\}$, whence $n\beta\neq1\beta=r\beta$ and so $r<n$. 

Again by Proposition \ref{chars}, we have $n\beta-1=|n\beta-1\beta|\in\{1,n-1\}$ and $(r+1)\beta-1=|(r+1)\beta-r\beta|\in\{1,n-1\}$, 
whence $n\beta, (r+1)\beta\in\{2,n\}$. 

Now, by Lemma \ref{ims}, there exist $1\leqslant a\leqslant b\leqslant n$ such that 
$[r+1,n]\beta=[a,b]$ or $[r+1,n]\beta=[b,n]\cup[1,a]$. 
Given the way we took $r$, we get $1\not\in[r+1,n]\beta$ and so we can only have $[r+1,n]\beta=[a,b]$. 

If $n\beta\neq (r+1)\beta$ then $2,n\in [r+1,n]\beta=[a,b]$, whence $[2,n]\subseteq\im(\beta)$. 
On the other hand, since we also have $1\beta=1$, we obtain $\im(\beta)=\Omega_n$ and so $\beta\in\Aut(C_n)$, 
which is a contradiction.  

Hence, $n\beta=(r+1)\beta$ and so, by Lemma \ref{eqim}, $n-(r+1)$ is even. 
Thus, $n-2=(n-(r+1))+(r-1)$ is even and so is $n$, as required. 
\end{proof}

Denote by $\mathscr{K}_n$ the subsemigroup of $\mathscr{T}_n$ constituted by its $n$ constant transformations.

\begin{proposition}\label{swend} 
For $n\geqslant4$, $\swEnd(C_n) = \sEnd(C_n)\cup\mathscr{K}_n$. 
In particular, 
for $n\geqslant5$, $\swEnd(C_n) = \Aut(C_n)\cup\mathscr{K}_n$. 
Moreover, in this case, for any $c\in\mathscr{K}_n$, $\{g,h,c\}$ is a generating set of $\swEnd(C_n)$  of minimum size and,  
consequently, $\swEnd(C_n)$ has rank three. 
\end{proposition}
\begin{proof}
We will use Proposition \ref{chars} several times without further mention.  

We begin by noticing that, clearly, $\sEnd(C_n)\cup\mathscr{K}_n\subseteq\swEnd(C_n)$. 

Conversely, let $\alpha\in\swEnd(C_n)\setminus\sEnd(C_n)$. 
Then, there exists $1\leqslant i\leqslant n$ such that $(i+1)\alpha=i\alpha$.  
We aim to show that $\alpha\in\mathscr{K}_n$. Admit, by contradiction, that $\alpha$ is not constant. 

Suppose that $1\alpha=n\alpha$. 
Since $\alpha$ is not constant, we can take the smallest $1<j<n$ such that $j\alpha\neq 1\alpha=n\alpha$.  
Then $(j-1)\alpha=1\alpha$ and so 
$|j\alpha-1\alpha|=|j\alpha-(j-1)\alpha|\in\{0,1,n-1\}$. Since $j\alpha\neq1\alpha$, we have 
$|j\alpha-1\alpha|\in\{1,n-1\}$ and so $j-1\in\{1,n-1\}$, i.e. $j=2$ or $j=n$. As $j\neq n$, it follows that $j=2$. 
Hence, 
$|n\alpha-2\alpha|=|1\alpha-2\alpha|\in\{0,1,n-1\}$ and, as $2\alpha\neq1\alpha$, we have 
$|n\alpha-2\alpha|\in\{1,n-1\}$ and so $n-2\in\{1,n-1\}$, which is a contradiction. 
Thus $1\alpha\neq n\alpha$ and so $1\leqslant i<n$. 

Suppose that $n\alpha=i\alpha$. Then, as $n\alpha\neq 1\alpha$, 
we have $|(i+1)\alpha-1\alpha|=|i\alpha-1\alpha|=|n\alpha-1\alpha|\in\{1,n-1\}$, 
whence $i,i-1\in\{1,n-1\}$, which is a contradiction. Thus $n\alpha\neq i\alpha$. 

Now, take the smallest $i+2\leqslant j\leqslant n$ such that $j\alpha\neq i\alpha$.  
Since $(j-1)\alpha=i\alpha\neq j\alpha$, we have 
$|j\alpha-(i+1)\alpha|=|j\alpha-i\alpha|=|j\alpha-(j-1)\alpha|\in\{1,n-1\}$, 
whence $j-i-1,j-i\in\{1,n-1\}$, which again is a contradiction. 

Therefore, $\alpha\in\mathscr{K}_n$ and so $\swEnd(C_n) = \sEnd(C_n)\cup\mathscr{K}_n$.  

For $n\geqslant5$, in view of Propositions \ref{send} and \ref{aut}, 
we also have $\swEnd(C_n) = \Aut(C_n)\cup\mathscr{K}_n= \mathscr{D}_{2n}\cup\mathscr{K}_n$. 
Since $\{g,h\}$ is a  generating set of $\mathscr{D}_{2n}$ of minimum size, it is clear that 
$\{g,h,c\}$ is a generating set of $\mathscr{D}_{2n}\cup\mathscr{K}_n$  of minimum size, for any $c\in\mathscr{K}_n$, 
as required.  
\end{proof} 

For $n\geqslant5$, it follows that $|\swEnd(C_n)| = 3n$.  
On the other hand, as $|\sEnd(C_4)|=32$, we have $|\swEnd(C_4)| =36$.  
Moreover, $\{g,h,u,c\}$, with $u$ any one of the $16$ elements of $\sEnd(C_4)$ with rank $3$ and $c\in\mathscr{K}_n$,  
constitutes a minimum size generating set of $\swEnd(C_4)$, whence $\swEnd(C_4)$ has rank four. 

\smallskip

Next, we exhibit formulas for the cardinalities of $\End(C_n)$ and $\wEnd(C_n)$. 
The formula for $|\End(C_n)|$ can also be found in the paper \cite{Michels&Knauer:2009} due to Michels and Knauer. 
Here, we give an alternative proof.

\begin{proposition}\label{sizes}
For $n\geqslant3$, 
$$
|\End(C_n)|=
\left\{\begin{array}{ll}
2n & \mbox{if $n$ is odd}\\
\displaystyle 2n+n\binom{n}{\frac{n}{2}} & \mbox{if $n$ is even}
\end{array}\right.
\quad\text{and}\quad
|\wEnd(C_n)|= 3n + 2n\sum_{k=1}^{\lfloor\frac{n}{2}\rfloor}\binom{2k-1}{k}\binom{n}{2k}. 
$$
\end{proposition}
\begin{proof} 
First, we prove the formula for $|\End(C_n)|$. If $n$ is odd then Propositions \ref{aut} and \ref{endodd} give us $\End(C_n)=\mathscr{D}_{2n}$ and so 
$|\End(C_n)|=2n$. Hence, suppose that $n$ is even. 

Let $\alpha\in\End(C_n)$. Then, by Proposition \ref{chars}, for all $1\leqslant i\leqslant n$, 
there exists $t\in\{1-n,-1,1,n-1\}$ such that $(i+1)\alpha=i\alpha+t$. 
Therefore, for $1\leqslant i\leqslant n-1$, we may take $t_i\in\{-1,1\}$ such that $(i+1)\alpha\equiv (i\alpha+t_i)\,(\text{mod}\,n)$. 
It follows that $n\alpha\equiv (1\alpha+\sum_{i=1}^{n-1}t_i)\,(\text{mod}\,n)$, i.e. $n\alpha-1\alpha\equiv \sum_{i=1}^{n-1}t_i\,(\text{mod}\,n)$. 
Since $1-n\leqslant\sum_{i=1}^{n-1}t_i\leqslant n-1$ and $n\alpha-1\alpha \in\{1-n,-1,1,n-1\}$, we deduce that 
$\sum_{i=1}^{n-1}t_i\in\{1-n,-1,1,n-1\}$. 

Let $T_{-1}=\{1\leqslant i\leqslant n-1\mid t_i=-1\}$ and $T_{1}=\{1\leqslant i\leqslant n-1\mid t_i=1\}$. 

If $|\sum_{i=1}^{n-1}t_i|=n-1$ then $(|T_{-1}|,|T_1|)\in\{(n-1,0),(0,n-1)\}$ and since $1\alpha$ can assume $n$ values, we have $2n$ elements in $\End(C_n)$ under these conditions. 

If $|\sum_{i=1}^{n-1}t_i|=1$ then $(|T_{-1}|,|T_1|)\in\{(\frac{n}{2}-1,\frac{n}{2}),(\frac{n}{2},\frac{n}{2}-1)\}$. Since there are $\binom{n-1}{\frac{n}{2}}$ choices for values of $t_i$ equal to $1$ 
(as well as equal to $-1$) and $n$ choices for the value of $1\alpha$, we have $2n\binom{n-1}{\frac{n}{2}}$ elements in $\End(C_n)$ under these conditions. 

Thus, we conclude that $|\End(C_n)|=2n+2n\binom{n-1}{\frac{n}{2}}=2n+n\binom{n}{\frac{n}{2}}$. 

\smallskip 

Next, we show the formula for $|\wEnd(C_n)|$, proceeding with a reasoning similar to the previous one. 

Let $\alpha\in\wEnd(C_n)$. In this case, in view of Proposition \ref{chars},  
there exists $t_i\in\{-1,0,1\}$ such that $(i+1)\alpha\equiv (i\alpha+t_i)\,(\text{mod}\,n)$, for $1\leqslant i\leqslant n-1$. 
Then $n\alpha\equiv (1\alpha+\sum_{i=1}^{n-1}t_i)\,(\text{mod}\,n)$, i.e. $n\alpha-1\alpha\equiv \sum_{i=1}^{n-1}t_i\,(\text{mod}\,n)$. 
Since $1-n\leqslant\sum_{i=1}^{n-1}t_i\leqslant n-1$ and, by Proposition \ref{chars}, $n\alpha-1\alpha \in\{1-n,-1,0,1,n-1\}$, it follows that 
$\sum_{i=1}^{n-1}t_i\in\{1-n,-1,0,1,n-1\}$. 

Let $T_{-1}=\{1\leqslant i\leqslant n-1\mid t_i=-1\}$ and $T_{1}=\{1\leqslant i\leqslant n-1\mid t_i=1\}$. 

If $|\sum_{i=1}^{n-1}t_i|=n-1$ then $(|T_{-1}|,|T_1|)\in\{(n-1,0),(0,n-1)\}$ and since $1\alpha$ can assume $n$ values, as in the case of $\End(C_n)$, 
we also have $2n$ elements in $\wEnd(C_n)$ under these conditions. 

If $\sum_{i=1}^{n-1}t_i=0$ then $|T_{-1}|=|T_1|=k$, for some $0\leqslant k\leqslant \lfloor\frac{n}{2}\rfloor$. 
Since for each $0\leqslant k\leqslant \lfloor\frac{n}{2}\rfloor$ there are $\binom{n-1}{k}\binom{n-1-k}{k}$ choices for non-null $t_i$'s and $1\alpha$ can assume $n$ values, 
we have $n\sum_{k=0}^{\lfloor\frac{n}{2}\rfloor}\binom{n-1}{k}\binom{n-1-k}{k}$ elements in $\wEnd(C_n)$ under these conditions. 

If $\sum_{i=1}^{n-1}t_i=-1$ then $|T_{-1}|=|T_1|+1=k$, for some $1\leqslant k\leqslant \lfloor\frac{n}{2}\rfloor$. 
If $\sum_{i=1}^{n-1}t_i=1$ then $|T_1|=|T_{-1}|+1=k$, for some $1\leqslant k\leqslant \lfloor\frac{n}{2}\rfloor$. Hence, 
in these two cases, for each $1\leqslant k\leqslant \lfloor\frac{n}{2}\rfloor$, 
we have $\binom{n-1}{k}\binom{n-1-k}{k-1}$ choices for non-null $t_i$'s. 
As again $1\alpha$ can assume $n$ values, in each of these cases, 
we obtain $n\sum_{k=1}^{\lfloor\frac{n}{2}\rfloor}\binom{n-1}{k}\binom{n-1-k}{k-1}$ elements in $\wEnd(C_n)$ under these conditions. 

Thus, we get $|\wEnd(C_n)|=2n + n\sum_{k=0}^{\lfloor\frac{n}{2}\rfloor}\binom{n-1}{k}\binom{n-1-k}{k} + 2 n\sum_{k=1}^{\lfloor\frac{n}{2}\rfloor}\binom{n-1}{k}\binom{n-1-k}{k-1} = 
3n + 2n\sum_{k=1}^{\lfloor\frac{n}{2}\rfloor}\binom{2k-1}{k}\binom{n}{2k}$, as required.  
\end{proof}

\section{Regularity and Green's relations of $\End(C_n)$ and $\wEnd(C_n)$}\label{regularity}

In order to describe the regular elements of $\wEnd(C_n)$ and $\End(C_n)$, we begin this section by defining the following concept.

Let $1\leqslant k\leqslant n$ and let $\alpha\in\wEnd(C_n)$ be a transformation of rank $k$. 
We say that $\alpha$ has a \textit{full sublist of consecutive images} if there exist $1\leqslant i\leqslant n$ and $u\in\{-1,1\}$ such that 
$(i+t)\alpha=i\alpha+tu$ for all $0\leqslant t\leqslant k-1$. 
In this case, 
being $I$ the arc $\{i, i+1,\ldots,i+k-1\}$ of $\Omega_n$, 
$\alpha_{|I}$ is injective and $I\alpha=\im(\alpha)$. 

Notice that, any transformation of $\wEnd(C_n)$ with rank $1$ or $2$ has a full sublist of consecutive images. 
Moreover, any transformation of $\End(C_n)$ with rank $3$ also has a full sublist of consecutive images.

\begin{proposition}\label{reg}
Let $n\geqslant3$ and $\alpha\in\wEnd(C_n)$. Then $\alpha$ is regular if and only if $\alpha$ has a full sublist of consecutive images. 
Moreover, if $\alpha\in\End(C_n)$ then $\alpha$ is regular in $\End(C_n)$ if and only if $\alpha$ is regular in $\wEnd(C_n)$. 
\end{proposition}
\begin{proof} 
Let us assume that $\alpha$ is a transformation of $\wEnd(C_n)$ with rank $k$ ($1\leqslant k\leqslant n$). 

First, suppose that $\alpha$ is regular (in $\wEnd(C_n)$). 
Let $\beta\in\wEnd(C_n)$ be such that $\alpha=\alpha\beta\alpha$. 
Then $\alpha$ is injective in $\im(\alpha\beta)$.  
As $\alpha\beta\in\wEnd(C_n)$, by Lemma \ref{ims}, we have $\im(\alpha\beta)=[a,b]$ or $\im(\alpha\beta)=[b,n]\cup[1,a]$, 
for some $1\leqslant a\leqslant b\leqslant n$. 
On the other hand, $k=|\im(\alpha)|=|\im(\alpha\beta\alpha)|\leqslant |\im(\alpha\beta)|\leqslant|\im(\alpha)|=k$, whence $|\im(\alpha\beta)|=k$. 

Let $i=a$, if $\im(\alpha\beta)=[a,b]$, and $i=b$, otherwise. Let $u=(i+1)\alpha-i\alpha$. 
Notice that, as $\alpha$ is injective in $\im(\alpha\beta)$, then $u\in\{-1,1\}$. 
Hence, $(i+t)\alpha=i\alpha+tu$ for $0\leqslant t\leqslant k-1$, 
i.e. $\alpha$ has a full sublist of consecutive images. 

Conversely, suppose there exist $1\leqslant i\leqslant n$ and $u\in\{-1,1\}$ such that 
$(i+t)\alpha=i\alpha+tu$ for all $0\leqslant t\leqslant k-1$. 

Let $\beta=g^{n+i-i\alpha}$, if $u=1$, and $\beta=hg^{i+i\alpha-1}$, otherwise. 
Then $\beta\in\wEnd(C_n)$ and $\alpha=\alpha\beta\alpha$, whence $\alpha$ is regular in $\wEnd(C_n)$. 

Since we also have $\beta\in\End(C_n)$, if $\alpha\in\End(C_n)$ then $\alpha$ is also regular in $\End(C_n)$. 

Thus, both statements of this result are proved. 
\end{proof}

Obviously, for all $n\geqslant 3$, $\Aut(C_n)$ is a regular semigroup. 
Then, by Propositions \ref{send} and \ref{endodd}, $\sEnd(C_n)$, for $n=3$ and $n\geqslant5$, and $\End(C_n)$, for all odd $n\geqslant3$, 
also are regular semigroups. Using Proposition \ref{reg}, we may show that $\sEnd(C_4)=\End(C_4)$, $\End(C_6)$, $\End(C_8)$, 
$\wEnd(C_3)$, $\wEnd(C_4)$ and $\wEnd(C_5)$ are regular semigroups. On the other hand, Proposition \ref{reg} also allows us to deduce that, 
for $n\geqslant5$, $\End(C_{2n})$ and $\wEnd(C_{n+1})$ are non-regular semigroups: for instance, 
$$
\left(\begin{array}{ccccccccccccc}
1&2&3&4&5&6&7&8&9&10&\cdots&2n-1&2n\\
1&2&3&2&3&4&3&2&3&2&\cdots&3&2
\end{array}\right)
\quad\text{and}\quad
\begin{pmatrix} 
1&2&3&4&5&6&\cdots&n+1\\
1&2&2&3&2&2&\cdots&2
\end{pmatrix}
$$
are non-regular elements of  $\End(C_{2n})$ and $\wEnd(C_{n+1})$, respectively. 
Finally, taking into account Proposition \ref{swend} and the equality $\swEnd(C_3)=\wEnd(C_3)$, it is easy to conclude that $\swEnd(C_n)$ is regular for $n\geqslant3$. 

Observe that these conclusions about the regularity of $\End(C_n)$ and $\wEnd(C_n)$ are, respectively, in agreement with Theorem 3.4 and Theorem 4.2 proved by Wilkeit in \cite{Wilkeit:1996}. 

\smallskip 

Next, in order to describe the Green's relations $\mathcal{R}$ on $\wEnd(C_{n})$ and on $\End(C_{n})$, we first present two lemmas. 

\begin{lemma}\label{2im}
Let $\alpha\in \wEnd(C_n)$. Let $1\leqslant j\leqslant n$ be such that $j,j+1\in\im(\alpha)$. 
Then there exists $i\in j\alpha^{-1}$ such that $\{i-1,i+1\}\cap(j+1)\alpha^{-1}\neq\emptyset$. 
\end{lemma}
\begin{proof} 
Let $p=\min{j\alpha^{-1}}$ and $q=\min{(j+1)\alpha^{-1}}$. 

\smallskip 

First, suppose that $p<q$ and let $i=\max\{x\in j\alpha^{-1}\mid x < q \}$. 
Then $i+1\leq q$ and, by Lemma \ref{ims}, there exist $1\leqslant a\leqslant b\leqslant n$ such that $[i+1,q]\alpha=[a,b]$ or $[i+1,q]\alpha=[b,n]\cup[1,a]$. 

Observe that $(i+1)\alpha\in\{i\alpha-1,i\alpha,i\alpha+1\}=\{j-1,j,j+1\}$. 
It follows by the maximality of $i$ that $j\not\in[i+1,q]\alpha$. 
In particular, $(i+1)\alpha\neq j$ and so $(i+1)\alpha\in\{j-1,j+1\}$. 

If $(i+1)\alpha=j+1$ then the lemma is proved. 

Hence, suppose that $(i+1)\alpha=j-1$. Then $j-1,j+1\in[i+1,q]\alpha$. 
Therefore, since $j\not\in[i+1,q]\alpha$, we cannot have $[i+1,q]\alpha=[a,b]$, whence $[i+1,q]\alpha=[b,n]\cup[1,a]$ 
and it is mandatory to have $a=j-1$ and $b=j+1$, i.e. $[i+1,q]\alpha=[1,j-1]\cup[j+1,n]$. 
As $j\in\im(\alpha)$ it follows that $\im(\alpha)=\Omega_n$ and so $\alpha\in\mathscr{D}_{2n}$. 
Then, as $(i-1)\alpha\in\{i\alpha-1,i\alpha+1\}=\{j-1,j+1\}$, $(i+1)\alpha=j-1$ and $\alpha$ is a permutation, we have $(i-1)\alpha=j+1$, which proves the lemma. 

\smallskip 

Next, suppose that $p>q$ and let $i=\min\{x\in j\alpha^{-1}\mid q<x \}$. 
Then, by a reasoning similar to the previous one, we can also conclude that $(i-1)\alpha=j+1$ or $(i+1)\alpha=j+1$, as required. 
\end{proof}

\begin{lemma}\label{sker}
Let $\alpha,\beta\in\wEnd(C_n)$ be such that $\im(\alpha)=\{1,2,\ldots,k\}$, $\im(\beta)=\{1,2,\ldots,\ell\}$ and $\ker(\beta)\subseteq\ker(\alpha)$, 
with $1\leqslant k\leqslant \ell\leqslant \lfloor\frac{n}{2}\rfloor+1$. 
Let $\phi:\{1,2,\ldots,\ell\}\longrightarrow\{1,2,\ldots,k\}$ be the surjective mapping defined by $j\phi=a$ if and only if $j\beta^{-1}\subseteq a\alpha^{-1}$, 
for $1\leqslant  j\leqslant \ell$ and $1\leqslant a\leqslant k$. Then $(j+1)\phi\in\{j\phi-1,j\phi,j\phi+1\}$, for $1\leqslant j\leqslant \ell-1$. 
Moreover, if $\alpha\in\End(C_n)$ then $(j+1)\phi\in\{j\phi-1, j\phi+1\}$, for $1\leqslant j\leqslant \ell-1$. 
\end{lemma} 
\begin{proof}
Let $1\leqslant j\leqslant \ell-1$. 
Then $j,j+1\in\im(\beta)$ and so, by Lemma \ref{2im}, there exists $i\in j\beta^{-1}$ such that $\{i-1,i+1\}\cap(j+1)\beta^{-1}\neq\emptyset$. 
Hence, as $j\beta^{-1}\subseteq (j\phi)\alpha^{-1}$ and $(j+1)\beta^{-1}\subseteq ((j+1)\phi)\alpha^{-1}$, we have $j\phi=i\alpha$ and, on the other hand, 
$(j+1)\phi=(i-1)\alpha$ or $(j+1)\phi=(i+1)\alpha$.  
Suppose that $(j+1)\phi=(i+1)\alpha$. Then $i\alpha=j\phi,(i+1)\alpha=(j+1)\phi\leqslant k<n$, whence we cannot have $|(i+1)\alpha-i\alpha|=n-1$.  
So, by Proposition \ref{chars}, we obtain $|(i+1)\alpha-i\alpha|\in\{0,1\}$ 
(and $|(i+1)\alpha-i\alpha|=1$, if $\alpha\in\End(C_n)$). Similarly, if $(j+1)\phi=(i-1)\alpha$ then 
$|(i-1)\alpha-i\alpha|\in\{0,1\}$ 
(and $|(i-1)\alpha-i\alpha|=1$, if $\alpha\in\End(C_n)$). 
Thus $(j+1)\phi\in\{j\phi-1,j\phi,j\phi+1\}$  (and $(j+1)\phi\in\{j\phi-1, j\phi+1\}$, if $\alpha\in\End(C_n)$), as required. 
\end{proof} 

Now, we can prove the following proposition. 

\begin{proposition} \label{kers} 
Let $\alpha,\beta\in\wEnd(C_n)$. Then $\ker(\alpha)=\ker(\beta)$ if and only if there exists $\sigma\in\mathscr{D}_{2n}$ such that $\alpha=\beta\sigma$. 
\end{proposition} 
\begin{proof}
First, notice that, it is clear that, if $\alpha=\beta\sigma$, for some $\sigma\in\mathscr{D}_{2n}$, then $\ker(\alpha)=\ker(\beta)$. 

Conversely, suppose that $\ker(\alpha)=\ker(\beta)$. Clearly, if $\alpha\in\mathscr{D}_{2n}$ then $\beta\in\mathscr{D}_{2n}$ and so $\alpha=\beta (\beta^{-1}\alpha)$. So, suppose that $\alpha\not\in\mathscr{D}_{2n}$ and let $k=\rank(\alpha)$. 
Then $1\leqslant  k \leqslant  \lfloor\frac{n}{2}\rfloor+1$, by Lemma \ref{wpar}. 
By (\ref{trans}), there exist $0\leqslant  r,s \leqslant  n-1$ such that $\im(\alpha g^r)=\{1,2,\ldots,k\}=\im(\beta g^s)$. 
Let $\alpha_1=\alpha g^r$ and $\beta_1=\beta g^s$. Then $\ker(\alpha_1)=\ker(\beta_1)$. 
Let $\phi:\{1,2,\ldots,k\}\longrightarrow\{1,2,\ldots,k\}$ be the surjective mapping defined by $j\phi=i$ if and only if $j\beta_1^{-1}\subseteq i\alpha_1^{-1}$, 
for $1\leqslant  i,j\leqslant k$. Then, by Lemma \ref{sker}, $(j+1)\phi\in\{j\phi-1,j\phi,j\phi+1\}$, for $1\leqslant j\leqslant k-1$. 
Furthermore, $\phi$ must be also injective, whence $(j+1)\phi\in\{j\phi-1,j\phi+1\}$, for $1\leqslant j\leqslant k-1$, and so  
$$
\phi=\phi_1= \begin{pmatrix} 
1&2&\cdots&k\\ 
1&2&\cdots&k
\end{pmatrix}
\quad\text{or}\quad 
\phi=\phi_2= \begin{pmatrix} 
1&2&\cdots&k\\
k&k-1&\cdots&1
\end{pmatrix}. 
$$
Now, take $\tau=1$ (in $\mathscr{D}_{2n}$), if $\phi=\phi_1$, and $\tau=hg^k$, if $\phi=\phi_2$. 
It is easy to show that $\alpha_1=\beta_1\phi=\beta_1\tau$ and so 
$\alpha=\alpha_1g^{n-r}=\beta_1\tau g^{n-r}=\beta g^s\tau g^{n-r}$, as required. 
\end{proof} 

An immediate corollary of Proposition \ref{kers} is: 

\begin{corollary}\label{Rrel} 
Let $\alpha,\beta\in\wEnd(C_n)$ {\rm[}respectively, $\alpha,\beta\in\End(C_n)${\rm]}. Then $\alpha\mathcal{R}\beta$ if and only $\ker(\alpha)=\ker(\beta)$. 
\end{corollary}

\smallskip

We could expect, given Proposition \ref{kers} and Corollary \ref{Rrel}, that two elements $\alpha$ and $\beta$ of $\wEnd(C_{n})$, or of $\End(C_{n})$, 
are $\mathcal{L}$-related if and only if there exists $\sigma \in \mathscr{D}_{2n}$ such that $\alpha=\sigma\beta$. This is not the case. 
For instance, in $\End(C_{12})$, take 
\small 
$$
\mbox{\normalsize$\alpha=$}\left(\begin{array}{cccccccccccc}
1&2&3&4&5&6&7&8&9&10&11&12\\ 
3&4&3&2&3&2&3&2&1&2&3&2
\end{array}\right) 
\mbox{\normalsize~and~}
\mbox{\normalsize$\beta=$}\left(\begin{array}{cccccccccccc}
1&2&3&4&5&6&7&8&9&10&11&12\\ 
3&2&1&2&1&2&3&2&3&4&3&2
\end{array}\right).
$$\normalsize 
It is easy to check that $\alpha\neq\sigma\beta$ for all $\sigma \in \mathscr{D}_{2\times12}$. 
Nevertheless, we have $\alpha \mathcal{L} \beta$ in $\End(C_{12})$ (and so in $\wEnd(C_{12})$). 
In fact,  for example, if 
\small 
$$
\mbox{\normalsize$\gamma=$}\left(\begin{array}{cccccccccccc}
1&2&3&4&5&6&7&8&9&10&11&12 \\ 
9&10&9&8&7&6&7&6&5&6&7&8
\end{array}\right) 
\mbox{\normalsize~and~}
\mbox{\normalsize$\lambda=$}\left(\begin{array}{cccccccccccccc}
1&2&3&4&5&6&7&8&9&10&11&12\\ 
11&10&9&10&9&10&11&12&1&2&1&12
\end{array}\right) 
$$\normalsize 
then 
$\alpha=\gamma\beta$ and $\beta=\lambda\alpha$. 

\medskip 

Next, we give characterizations of the $\mathcal{L}$-relations on $\End(C_{n})$ and on $\wEnd(C_{n})$, 
starting by proving the following lemma: 

\begin{lemma}\label{restinj}
Let $I$ be an arc of $\Omega_n$ and let $\alpha\in \wEnd(C_{n})$ be such that $\alpha_{|I}$ is injective. 
Then there exists $\sigma \in \mathscr{D}_{2n}$ such that $\alpha_{|I}=\sigma_{|I}$.
\end{lemma}
\begin{proof}
Since $\alpha$ is injective in $I$, if $I=[i,j]$, for some $1\leqslant i\leqslant j\leqslant n$, then 
$$
\alpha_{|I}=
\begin{pmatrix} 
i&i+1&\cdots&j\\ 
i\alpha&i\alpha+1&\cdots&i\alpha+j-i
\end{pmatrix}
=g^{i\alpha-i}_{|I}
\quad\text{or}\quad
\alpha_{|I}=\begin{pmatrix} 
i&i+1&\cdots&j\\ 
i\alpha&i\alpha-1&\cdots&i\alpha-j+i
\end{pmatrix}
=hg^{i\alpha+i-1}_{|I},
$$ 
and if $I=[j,n]\cup[1,i]$, for some $1\leqslant i< j-1\leqslant n-1$, then 
$$
\alpha_{|I}=
\begin{pmatrix} 
j&\cdots&n&1&\cdots&i\\ 
i\alpha-i-n+j&\cdots&i\alpha-i&i\alpha-i+1&\cdots&i\alpha
\end{pmatrix}
=g^{i\alpha-i}_{|I}
$$
or
$$
\alpha_{|I}=\begin{pmatrix} 
j&\cdots&n&1&\cdots&i\\ 
i\alpha+i+n-j&\cdots&i\alpha+i&i\alpha+i-1&\cdots&i\alpha
\end{pmatrix}
=hg^{i\alpha+i-1}_{|I},
$$ 
as required. 
\end{proof}

\begin{proposition}\label{Lrel}
Let $\alpha,\beta \in \wEnd(C_{n})$ {\rm[}respectively, $\alpha,\beta\in\End(C_n)${\rm]}. 
Then $\alpha \mathcal{L} \beta$ if and only if there exist an arc $I$ of $\Omega_n$, $\sigma \in \mathscr{D}_{2n}$ 
and idempotents $\varepsilon_{1}, \varepsilon_{2}\in \wEnd(C_{n})$ {\rm[}respectively, $\varepsilon_1,\varepsilon_2\in\End(C_n)${\rm]} 
such that $I\alpha=\im(\alpha)$, $\alpha_{|I}=(\sigma\beta)_{|I}$, $\im(\varepsilon_{1})=I\sigma$, $\im(\varepsilon_{2})=I$, 
$\alpha=\sigma\varepsilon_{1}\beta$ and  $\beta=\sigma^{-1}\varepsilon_{2}\alpha$. 
\end{proposition} 
\begin{proof}
Let $M\in\{\End(C_{n}), \wEnd(C_{n})\}$ and suppose that  $\alpha \mathcal{L} \beta$ in $M$. 
Then, there exist $\gamma, \lambda \in M$ such that $\alpha=\gamma\beta$ and $\beta=\lambda\alpha$.  
Hence, $\alpha={(\gamma\lambda)}^{k}\gamma\beta$ and $\beta={(\lambda\gamma)}^{k}\lambda\alpha$ for all integer $k\geqslant0$. 
Take a positive integer $\omega$ such that 
${(\gamma\lambda)}^{\omega}$ and ${(\lambda\gamma)}^{\omega}$ are both idempotents. 
Then, we have $\alpha={(\gamma\lambda)}^{\omega}\gamma\beta$,  
$\beta={(\lambda\gamma)}^{\omega}\lambda\alpha$  
and, since ${(\gamma\lambda)}^{\omega}\gamma={(\gamma\lambda)}^{\omega}{(\gamma\lambda)}^{\omega}{(\gamma\lambda)}^{\omega}\gamma=
{(\gamma\lambda)}^{\omega}\gamma\lambda{(\gamma\lambda)}^{\omega-1}{(\gamma\lambda)}^{\omega}\gamma$, 
the element $\gamma'={(\gamma\lambda)}^{\omega}\gamma=\gamma{(\lambda\gamma)}^{\omega}$ of $M$ is regular.

Now, by Proposition \ref{reg}, there exists an arc $I$ such that $\gamma'_{|I}$ is injective and $I\gamma'=\im(\gamma')$. 
Hence, by Lemma \ref{restinj}, there exists $\sigma\in \mathscr{D}_{2n}$ such that ${\sigma_{|}}_{I}={\gamma'_{|}}_{I}$. 
Thus, 
$I\alpha=I(\gamma'\beta)=(I\gamma')\beta=\im(\gamma')\beta=\im(\gamma'\beta)=\im(\alpha)$ and 
$\alpha_{|I}=(\gamma'\beta)_{|I} = \gamma'_{|I}\beta_{|I\gamma'}=\sigma_{|I}\beta_{|I\sigma}=(\sigma\beta)_{|I}$.

Next, we show that $\gamma'\sigma^{-1}\gamma'=\gamma'$. 
Let $x\in \Omega_{n}$. 
Then $x\gamma'\in \im(\gamma')=I\gamma'=I\sigma$, whence $x\gamma'\sigma^{-1} \in  I$ and so 
$x(\gamma'\sigma^{-1}\gamma')=(x\gamma'\sigma^{-1})\gamma'=(x\gamma'\sigma^{-1})\sigma=x\gamma'$. 
Thus, $\gamma'\sigma^{-1}\gamma'=\gamma'$. 

Let $\varepsilon_{1}=\sigma^{-1}\gamma'$. 
Then, $\varepsilon_{1}$ is clearly an idempotent, $\im(\varepsilon_{1})=\im(\gamma')=I\sigma$ and $\sigma\varepsilon_{1}\beta=\sigma\sigma^{-1}\gamma'\beta=\alpha$.

On the other hand, let $\varepsilon_{2}=\sigma{(\lambda\gamma)}^{\omega}\sigma^{-1}$. 
It is also clear that $\varepsilon_{2}$ is an idempotent and we have 
\begin{multline*}
\beta=
{(\lambda\gamma)}^{\omega}\lambda\alpha=
{(\lambda\gamma)}^{\omega}{(\lambda\gamma)}^{\omega}\lambda\alpha=
{(\lambda\gamma)}^{\omega-1}\lambda\gamma{(\lambda\gamma)}^{\omega}\lambda\alpha=
{(\lambda\gamma)}^{\omega-1}\lambda\gamma'\lambda\alpha=
{(\lambda\gamma)}^{\omega-1}\lambda\gamma'\sigma^{-1}\gamma'\lambda\alpha=\\
={(\lambda\gamma)}^{\omega-1}\lambda\gamma{(\lambda\gamma)}^{\omega}\sigma^{-1}\gamma{(\lambda\gamma)}^{\omega}\lambda\alpha=
{(\lambda\gamma)}^{\omega}\sigma^{-1}\gamma{(\lambda\gamma)}^{\omega}\lambda\alpha=
{(\lambda\gamma)}^{\omega}\sigma^{-1}\gamma \beta =
{(\lambda\gamma)}^{\omega}\sigma^{-1}\alpha =\\
=\sigma^{-1}\sigma{(\lambda\gamma)}^{\omega}\sigma^{-1}\alpha =
\sigma^{-1}\varepsilon_2\alpha. 
\end{multline*}
Moreover, $\im({(\lambda\gamma)}^{\omega})=\im(\gamma')$. 
In fact, ${(\lambda\gamma)}^{\omega}={(\lambda\gamma)}^{\omega}{(\lambda\gamma)}^{\omega}=
{(\lambda\gamma)}^{\omega-1}\lambda(\gamma{(\lambda\gamma)}^{\omega})={(\lambda\gamma)}^{\omega-1}\lambda\gamma'$, 
which implies that 
$\im({(\lambda\gamma)}^{\omega})\subseteq \im(\gamma')$ and, 
as $\gamma'=\gamma{(\lambda\gamma)}^{\omega}$, we also obtain $\im(\gamma')\subseteq \im({(\lambda\gamma)}^{\omega})$. 
Therefore, finally, we have 
$\im(\varepsilon_{2})=\im (\sigma{(\lambda\gamma)}^{\omega}\sigma^{-1})=\im({(\lambda\gamma)}^{\omega}\sigma^{-1})=
\im({(\lambda\gamma)}^{\omega})\sigma^{-1}=\im(\gamma')\sigma^{-1}=(I\sigma)\sigma^{-1}=I$.

\smallskip 

Since the converse implication is obviously true, the proof is finished. 
\end{proof}

\medskip 

Going back to the above example, take $I=\{9,10,11,12,1,2\}$, $\sigma=g^8$, \small 
$$
\mbox{\normalsize$\varepsilon_{1}=$}\left(\begin{array}{cccccccccccc}
1&2&3&4&5&6&7&8&9&10&11&12\\
7&6&7&6&5&6&7&8&9&10&9&8
\end{array}\right)
\mbox{\normalsize~and~}
\mbox{\normalsize$\varepsilon_{2}=$}\left(\begin{array}{cccccccccccc}
1&2&3&4&5&6&7&8&9&10&11&12\\ 
1&2&1&12&11&10&9&10&9&10&11&12
\end{array}\right). 
$$\normalsize 
Then, we get 
$I\alpha=\im(\alpha)$, $\alpha_{|I}=(\sigma\beta)_{|I}$, $\im(\varepsilon_{1})=I\sigma$, $\im(\varepsilon_{2})=I$, 
$\alpha=\sigma\varepsilon_{1}\beta$ and $\beta=\sigma^{-1}\varepsilon_{2}\alpha$. 

\medskip 

We finish this section with descriptions of the $\mathcal{D}$-relation on $\End(C_{n})$ and $\wEnd(C_{n})$, which proofs are a straightforward consequence of 
Propositions \ref{kers} and \ref{Lrel}. 

\begin{proposition}
Let $\alpha,\beta \in \wEnd(C_{n})$ {\rm[}respectively, $\alpha,\beta\in\End(C_n)${\rm]}. 
Then $\alpha \mathcal{D} \beta$ if and only if there exists an arc $I$ of $\Omega_n$, $\sigma,\tau \in \mathscr{D}_{2n}$ 
and idempotents $\varepsilon_{1}, \varepsilon_{2}\in \wEnd(C_{n})$ {\rm[}respectively, $\varepsilon_{1}, \varepsilon_{2}\in\End(C_n)${\rm]} 
such that $I\alpha=\im(\alpha)$, 
$\alpha_{|I}=(\sigma\beta\tau)_{|I}$, 
$\im(\varepsilon_{1})=I\sigma$, 
$\im(\varepsilon_{2})=I$, 
$\alpha=\sigma\varepsilon_{1}\beta\tau$ and  $\beta=\sigma^{-1}\varepsilon_{2}\alpha\tau^{-1}$. 
\end{proposition}

\section{On the ranks of $\End(C_n)$ and $\wEnd(C_n)$} \label{ranksec}

Let $\mathscr{X}$ [respectively, $\mathscr{Y}$] be a transversal for the relation $\mathcal{R}$ in $\wEnd(C_n)\setminus\mathscr{D}_{2n}$ 
[respectively, in $\End(C_n)\setminus\mathscr{D}_{2n}$]. 
Then, it follows another immediate corollary of Proposition \ref{kers}.   
 
\begin{corollary}\label{gens} 
The semigroups $\wEnd(C_n)$ and $\End(C_n)$ are generated by  $\mathscr{X}\cup\{g,h\}$ and $\mathscr{Y}\cup\{g,h\}$, respectively. 
\end{corollary}

Observe that we can choose $\mathscr{Y}=\mathscr{X}\cap\End(C_n)$. 

Now, let $S\in\{\End(C_n),\wEnd(C_n)\}$. Define an equivalence relation $\sim$ on $S$ by $\alpha\sim\beta$ if and only if 
there exist $\sigma,\xi\in\mathscr{D}_{2n}$ such that $\alpha=\sigma\beta\xi$, for $\alpha,\beta\in S$. 
Clearly, for $\alpha,\beta\in S$, if $\alpha\sim\beta$ then $\rank(\alpha)=\rank(\beta)$. 

Next, define a quasi-order relation $\preccurlyeq$ on $S$ by $\alpha\preccurlyeq\beta$ if and only 
if there exist $\alpha'\in [\alpha]_\sim$ and $\beta'\in [\beta]_\sim$ such that $\ker(\alpha')\subseteq\ker(\beta')$, for $\alpha,\beta\in S$. 

It is easy to show that $\sim$ is the equivalence relation on $S$ induced by $\preccurlyeq$, i.e. $\alpha\sim\beta$ if and only if $\alpha\preccurlyeq\beta$ and $\beta\preccurlyeq\alpha$, 
for $\alpha,\beta\in S$.  

Let $\tilde{\mathscr{X}}$  and $\tilde{\mathscr{Y}}$ be transversals for the relation $\sim$ in $\mathscr{X}$ and $\mathscr{Y}$, respectively. 
Then, clearly, we have:  

\begin{lemma}\label{tils} 
The semigroups $\wEnd(C_n)$ and $\End(C_n)$ are generated by  $\tilde{\mathscr{X}}\cup\{g,h\}$ and $\tilde{\mathscr{Y}}\cup\{g,h\}$, respectively. 
\end{lemma}
 
Let us consider the following algorithm:

\begin{algorithm} \label{algo} 
\begin{description}

\item[\em Input.] Let $\mathscr{Z}$ be a subset of $\wEnd(C_n)\setminus\mathscr{D}_{2n}$. 

\item[\em Step $0$.] For $1\leqslant k\leqslant {\lfloor\frac{n}{2}\rfloor}+1$, define $\mathscr{Z}_k:=\{\alpha\in \mathscr{Z}\mid \rank(\alpha)=k\}$. 

\item[\em Step $1$.] Define $\mathscr{W}_1:=\mathscr{Z}_{{\lfloor\frac{n}{2}\rfloor}+1}$. 

\item[\em Step $k$.] {\rm[}$2\leqslant k\leqslant {\lfloor\frac{n}{2}\rfloor}+1${\rm]} Define 
$
\mathscr{W}_{k}:=\mathscr{Z}_{{\lfloor\frac{n}{2}\rfloor}-k+2}\setminus\left\{\alpha\in\mathscr{Z}_{{\lfloor\frac{n}{2}\rfloor}-k+2}\mid 
\mbox{$\beta\preccurlyeq\alpha$, for some $\beta\in\displaystyle\bigcup_{i=1}^{k-1}\mathscr{W}_i$}\right\}. 
$
\item[\em Output.] Let $\mathscr{W}=\displaystyle\bigcup_{i=1}^{{\lfloor\frac{n}{2}\rfloor}+1}\mathscr{W}_i$. 
\end{description}
\end{algorithm}
Let $\hat{\mathscr{X}}$  and $\hat{\mathscr{Y}}$ be the sets resulting from applying Algorithm \ref{algo} to $\tilde{\mathscr{X}}$ and $\tilde{\mathscr{Y}}$, respectively.

\begin{lemma}\label{fund} 
Let $\alpha,\beta\in\wEnd(C_n)$ be such that $\im(\alpha)=\{1,2,\ldots,k-1\}$, $\im(\beta)=\{1,2,\ldots,\ell\}$ and $\ker(\beta)\subseteq\ker(\alpha)$, 
with $2\leqslant k\leqslant \ell\leqslant \lfloor\frac{n}{2}\rfloor+1$. 
\begin{enumerate}
\item If $\alpha\in\End(C_n)$   {\em(}and so $n$ must be even{\em)} then there exists $\gamma\in\End(C_n)$  such that $\rank(\gamma)=k$ and $\alpha=\beta\gamma$.  

\item If $k<\lfloor\frac{n}{2}\rfloor+1$ then there exists $\gamma\in\wEnd(C_n)$  
such that $\rank(\gamma)=k$ and $\alpha=\beta\gamma$. 

\item If $k=\lfloor\frac{n}{2}\rfloor+1$ {\em(}and so $n$ must be even{\em)} and $\alpha\not\in\End(C_n)$ 
then there exist $\lambda\in \hat{\mathscr{X}}$ and $\eta,\zeta\in \mathscr{D}_{2n}$ such that 
$\alpha=\beta\eta\lambda\zeta$. 
\end{enumerate} 
\end{lemma}
\begin{proof}
Let $\phi:\{1,2,\ldots,\ell\}\longrightarrow\{1,2,\ldots,k-1\}$ be the surjective mapping defined by 
$j\phi=i$ if and only if $j\beta^{-1}\subseteq i\alpha^{-1}$, for $1\leqslant  j\leqslant \ell$ and $1\leqslant i\leqslant k-1$. 
Then, by Lemma \ref{sker}, we have 
\begin{equation}\label{phi}
\mbox{$(j+1)\phi\in\{j\phi-1,j\phi,j\phi+1\}$ (and $(j+1)\phi\in\{j\phi-1, j\phi+1\}$, if $\alpha\in\End(C_n)$), for $1\leqslant j\leqslant \ell-1$.}
\end{equation}

\smallskip

Let us first consider the case $k<\lfloor\frac{n}{2}\rfloor+1$. We will consider two subcases. 

\smallskip 

First, suppose that $1\phi\leqslant  k-\ell\phi$. Then 
$$\textstyle 
\ell+\ell\phi+1\phi-1\leqslant \ell+\ell\phi+(k-\ell\phi)-1=\ell+k-1<(\lfloor\frac{n}{2}\rfloor+1)+(\lfloor\frac{n}{2}\rfloor+1)-1\leqslant n+1 
$$
and so define 
\begin{multline}\nonumber 
\gamma=\left(\begin{array}{ccc|ccccccc|}
1&\cdots&\ell&\ell+1&\cdots&\ell+\ell\phi-1&\ell+\ell\phi&\ell+\ell\phi+1&\cdots&\ell+\ell\phi+1\phi-1\\ 
1\phi&\cdots&\ell\phi&\ell\phi-1&\cdots&1&n&1&\cdots&1\phi-1
\end{array}\right.
\\
\left.\begin{array}{|ccc}
\ell+\ell\phi+1\phi&\cdots&n\\
1\phi&\cdots&* 
\end{array}\right), 
\end{multline}
where the images of $[\ell+\ell\phi+1\phi,n]$ alternate between $1\phi$ and $1\phi-1$, in case of $\ell+\ell\phi+1\phi\leqslant n$  
(so $n\gamma=1\phi-1$, if $n-\ell-\ell\phi-1\phi+1$ is even, and  $n\gamma=1\phi$, otherwise). 

In view of (\ref{phi}), it is clear that $\gamma\in\wEnd(C_n)$ and $\rank(\gamma)=k$. Moreover, from $j\beta^{-1}\subseteq (j\phi)\alpha^{-1}$, for $1\leqslant  j\leqslant \ell$, 
it follows immediately that $\alpha=\beta\gamma$. 

On the other hand, suppose that $\alpha\in\End(C_n)$. Then $n$ must be even and, by (\ref{phi}), if $\ell$ is odd then $1\phi$ is odd if and only if $\ell\phi$ is odd, 
and if $\ell$ is even then $1\phi$ is odd if and only if $\ell\phi$ is even, whence we obtain, in both cases, that $\ell+\ell\phi+1\phi-1$ is even. 
Hence $n\gamma=1\phi-1$ and so we may conclude that $\gamma\in\End(C_n)$. 

\smallskip 

Secondly, suppose that $1\phi> k-\ell\phi$. Then 
$$\textstyle 
\ell+k-\ell\phi+k-1\phi-1 <   \ell+k-\ell\phi+k-k+\ell\phi-1 = \ell+k-1<(\lfloor\frac{n}{2}\rfloor+1)+(\lfloor\frac{n}{2}\rfloor+1)-1\leqslant n+1 
$$
and so define 
\begin{multline}\nonumber 
\gamma=\left(\begin{array}{ccc|cccccc|}
1&\cdots&\ell&\ell+1&\cdots&\ell+k-\ell\phi&\ell+k-\ell\phi+1&\cdots&\ell+k-\ell\phi+k-1\phi-1\\ 
1\phi&\cdots&\ell\phi&\ell\phi+1&\cdots&k&k-1&\cdots&1\phi+1
\end{array}\right.
\\
\left.\begin{array}{|ccc}
\ell+k-\ell\phi+k-1\phi&\cdots&n\\
1\phi&\cdots&* 
\end{array}\right), 
\end{multline}
where the images of $[\ell+k-\ell\phi+k-1\phi,n]$ alternate between $1\phi$ and $1\phi+1$, in case of $\ell+k-\ell\phi+k-1\phi\leqslant n$  
(so $n\gamma=1\phi+1$, if $n-\ell-2k+\ell\phi+1\phi+1$ is even, and  $n\gamma=1\phi$, otherwise). 

Clearly, by (\ref{phi}), we have $\gamma\in\wEnd(C_n)$ and $\rank(\gamma)=k$. 
As in the first subcase, by the same argument, we also obtain $\alpha=\beta\gamma$. 

Now, suppose that $\alpha\in\End(C_n)$. Then $n$ is even. 
On the other hand, as in the previous subcase, we may conclude that $\ell-\ell\phi-1\phi-1$ is even and so that 
$\ell+k-\ell\phi+k-1\phi-1$ is also even. 
Hence $n\gamma=1\phi+1$ and so we can also conclude that $\gamma\in\End(C_n)$. 

\medskip 

Next, we consider the case $k=\lfloor\frac{n}{2}\rfloor+1$. Then $n$ is even and $k=\ell=\frac{n}{2}+1$. 
Since $\phi:\{1,2,\ldots,\frac{n}{2}+1\}\longrightarrow\{1,2,\ldots,\frac{n}{2}\}$ is a surjective mapping, 
all kernel classes of $\phi$ are singletons except exactly one that has precisely two elements. 

Observe that, as $\rank(\beta)=\frac{n}{2}+1$, by Lemma \ref{wpar}, $\beta\in\End(C_n)$. 

Again, we will consider two subcases. 

\smallskip 

First, suppose that $(j+1)\phi\in\{j\phi-1, j\phi+1\}$ for all $1\leqslant j\leqslant \ell-1$.  
Then the unique non-singleton kernel class of $\phi$ must be $\{1,3\}$ or $\{\frac{n}{2}-1,\frac{n}{2}+1\}$. 
Furthermore, $\phi$ must be one of the following four mappings: 
$$
\begin{pmatrix} 
1&2&3&\cdots&\frac{n}{2}+1\\
2&1&2&\cdots&\frac{n}{2}
\end{pmatrix}, 
\begin{pmatrix} 
1&2&3&\cdots&\frac{n}{2}+1\\
\frac{n}{2}-1&\frac{n}{2}&\frac{n}{2}-1&\cdots&1
\end{pmatrix}, 
$$
$$
\begin{pmatrix} 
1&\cdots&\frac{n}{2}-1&\frac{n}{2}&\frac{n}{2}+1\\
1&\cdots&\frac{n}{2}-1&\frac{n}{2}&\frac{n}{2}-1
\end{pmatrix}
\quad\text{or}\quad 
\begin{pmatrix} 
1&\cdots&\frac{n}{2}-1&\frac{n}{2}&\frac{n}{2}+1\\
\frac{n}{2}&\cdots&2&1&2
\end{pmatrix}. 
$$
Therefore, by taking, respectively,
$$
\gamma=\left(\begin{array}{ccccc|cccc}
1&2&3&\cdots&\frac{n}{2}+1   &\frac{n}{2}+2&\frac{n}{2}+3&\cdots&n\\
2&1&2&\cdots&\frac{n}{2}       &\frac{n}{2}+1 & \frac{n}{2} & \cdots & 3
\end{array}\right), 
$$
$$
\gamma=\left(\begin{array}{ccccc|cccc}
1&2&3&\cdots&\frac{n}{2}+1   &\frac{n}{2}+2&\frac{n}{2}+3&\cdots&n\\
\frac{n}{2}-1&\frac{n}{2}&\frac{n}{2}-1&\cdots&1 & n & 1 & \cdots & \frac{n}{2}-2
\end{array}\right), 
$$
$$
\gamma=\left(\begin{array}{ccccc|ccccc}
1&\cdots&\frac{n}{2}-1&\frac{n}{2}&\frac{n}{2}+1    &\frac{n}{2}+2&\frac{n}{2}+3&\cdots&n-1&n\\
1&\cdots&\frac{n}{2}-1&\frac{n}{2}&\frac{n}{2}-1    &\frac{n}{2}-2&\frac{n}{2}-3&\cdots&1&n
\end{array}\right)
\quad\text{and}\quad 
$$
$$
\gamma=\left(\begin{array}{ccccc|ccc}
1&\cdots&\frac{n}{2}-1&\frac{n}{2}&\frac{n}{2}+1   & \frac{n}{2}+2 & \cdots & n\\
\frac{n}{2}&\cdots&2&1&2   & 3 &\cdots &\frac{n}{2}+1
\end{array}\right), 
$$
we have $\gamma\in\End(C_n)$, $\rank(\gamma)=\frac{n}{2}+1=k$ and $\alpha=\beta\gamma$. 

Notice that, as $\beta,\gamma\in\End(C_n)$, this subcase occurs if and only if $\alpha\in\End(C_n)$. 

\smallskip 

Finally, we consider the opposite subcase, i.e.  there exists $1\leqslant i\leqslant \ell-1\leqslant\frac{n}{2}$ such that $(i+1)\phi\not\in\{i\phi-1, i\phi+1\}$. 
Then, by (\ref{phi}), we have $(i+1)\phi=i\phi$.
Notice that, in this case, $\alpha\not\in\End(C_n)$.  
As all kernel classes of $\phi$ are singletons except exactly one, 
we have $(j+1)\phi\in\{j\phi-1, j\phi+1\}$, for all $j\in\{1,2,\ldots,\ell-1\}\setminus\{i\}$. 
Furthermore, $\phi$ must be one of the following mappings: 
$$
\phi_i=\begin{pmatrix} 
1&\cdots&i&i+1&i+2&\cdots&\frac{n}{2}+1\\
1&\cdots&i&i&i+1&\cdots&\frac{n}{2}
\end{pmatrix}
~\text{or}~
\phi'_i=\begin{pmatrix} 
1&\cdots&i&i+1&\cdots&\frac{n}{2}+1\\
\frac{n}{2}&\cdots&\frac{n}{2}-i+1&\frac{n}{2}-i+1&\cdots&1
\end{pmatrix}. 
$$
Now, define 
$$
\gamma_i=\left(\begin{array}{ccccccc|ccc}
1&\cdots&i&i+1&i+2&\cdots&\frac{n}{2}+1    & \frac{n}{2}+2 & \cdots & n\\
1&\cdots&i&i&i+1&\cdots&\frac{n}{2}       & \frac{n}{2} & \cdots & 2
\end{array}\right)
$$
and 
$$
\gamma'_i=\left(\begin{array}{cccccc|cccc}
1&\cdots&i&i+1&\cdots&\frac{n}{2}+1     & \frac{n}{2}+2 & \cdots & n \\
\frac{n}{2}&\cdots&\frac{n}{2}-i+1&\frac{n}{2}-i+1&\cdots&1 & 1 & \cdots & \frac{n}{2}-1
\end{array}\right), 
$$
if $1\leqslant i<\frac{n}{2}$, and 
$$
\gamma_i=\left(\begin{array}{cccc|ccc}
1&\cdots& \frac{n}{2} & \frac{n}{2}+1    & \frac{n}{2}+2 & \cdots & n\\
1&\cdots&\frac{n}{2} & \frac{n}{2}      & \frac{n}{2}-1 & \cdots & 1
\end{array}\right)
\quad\text{and}\quad 
\gamma'_i=\left(\begin{array}{ccccc|ccc}
1&\cdots&\frac{n}{2}-1&\frac{n}{2}&\frac{n}{2}+1     & \frac{n}{2}+2 & \cdots & n \\
\frac{n}{2}&\cdots&2&1 & 1 &2&\cdots & \frac{n}{2}
\end{array}\right), 
$$ 
if $i=\frac{n}{2}$. 
Clearly, $\gamma_i,\gamma'_i\in\wEnd(C_n)$ and $\rank(\gamma_i)=\rank(\gamma'_i)=\frac{n}{2}$. 
Moreover, if $\phi=\phi_i$ then $\alpha=\beta\gamma_i$ and if $\phi=\phi'_i$ then $\alpha=\beta\gamma'_i$. 

Let $\gamma\in\{\gamma_i, \gamma'_i\}$ be such that $\alpha=\beta\gamma$. 

Since $\gamma$ has two distinct kernel classes with two consecutive (modulo $n$) elements and, by Lemma \ref{wpar}, 
any element of $\wEnd(C_n)$ with rank equal to $\frac{n}{2}+1$ belongs to $\End(C_n)$, we must have 
\begin{equation}\label{nsub}
\ker(\delta)\not\subseteq\ker(\gamma),
\end{equation}
for all $\delta\in\wEnd(C_n)$ such that $\rank(\delta)=\frac{n}{2}+1$. 

Let $\tau_1\in\mathscr{X}$ be such that $\ker(\tau_1)=\ker(\gamma)$. 
Take $\sigma\in\mathscr{D}_{2n}$ such that $\gamma=\tau_1\sigma$. 

Let $\tau_2\in\tilde{\mathscr{X}}$ be such that $\tau_2\sim\tau_1$. 
Take $\sigma_1,\sigma_2\in\mathscr{D}_{2n}$ such that $\tau_1=\sigma_1\tau_2\sigma_2$. 

Suppose, by contradiction, that $\tau_2\not\in\hat{\mathscr{X}}$. 
Then, as $\rank(\tau_2)=\rank(\tau_1)=\rank(\gamma)=\frac{n}{2}$, 
there exists $\delta\in \hat{\mathscr{X}}$ such that $\delta\preccurlyeq\tau_2$ and $\rank(\delta)=\frac{n}{2}+1$. 

Let $\delta'\in[\delta]_\sim$ and $\tau'_2\in[\tau_2]_\sim$ be such that $\ker(\delta')\subseteq\ker(\tau'_2)$. 
Let $\xi_1,\xi_2\in\mathscr{D}_{2n}$ be such that $\tau_2=\xi_1\tau'_2\xi_2$. 

Hence, $\tau_1=\sigma_1\tau_2\sigma_2 = \sigma_1\xi_1\tau'_2\xi_2\sigma_2$ and so 
$\ker(\sigma_1\xi_1\delta'\xi_2\sigma_2)\subseteq \ker(\sigma_1\xi_1\tau'_2\xi_2\sigma_2)=\ker(\tau_1)=\ker(\gamma)$. 
Now, since $\rank(\sigma_1\xi_1\delta'\xi_2\sigma_2)=\rank(\delta')=\rank(\delta)=\frac{n}{2}+1$, by (\ref{nsub}), we obtain a contradiction. 

Thus $\tau_2\in\hat{\mathscr{X}}$. 

Finally, as $\gamma=\tau_1\sigma=\sigma_1\tau_2\sigma_2\sigma$, we have $\alpha=\beta\gamma=\beta \sigma_1\tau_2\sigma_2\sigma$, 
as required. 
 \end{proof}

Now, we can prove our main result:

\begin{theorem}\label{ranks}
The sets $\hat{\mathscr{X}}\cup\{g,h\}$ and $\hat{\mathscr{Y}}\cup\{g,h\}$ are generating sets with minimum size of $\wEnd(C_n)$ and $\End(C_n)$, respectively. 
In particular, $\wEnd(C_n)$ and $\End(C_n)$ have ranks $|\hat{\mathscr{X}}|+2$ and $|\hat{\mathscr{Y}}|+2$, respectively. 
\end{theorem}
\begin{proof}
We divide this proof in four steps. 

\smallskip 

\noindent\textsc{step 1}. 
First, we prove that $\hat{\mathscr{X}}\cup\{g,h\}$ is a generating set of $\wEnd(C_n)$. 
Since $\tilde{\mathscr{X}}\cup\{g,h\}$ generates $\wEnd(C_n)$, by Lemma \ref{tils}, 
it suffices to show that $\alpha\in \langle\hat{\mathscr{X}}\cup\{g,h\}\rangle$, 
for all $\alpha\in\tilde{\mathscr{X}}$. 

Observe that, by construction (Algorithm \ref{algo}), 
$\{\alpha\in \tilde{\mathscr{X}}\mid \rank(\alpha)={\lfloor\frac{n}{2}\rfloor}+1\}=\{\alpha\in \hat{\mathscr{X}}\mid \rank(\alpha)={\lfloor\frac{n}{2}\rfloor}+1\}$. 
Hence, trivially,  $\alpha\in \langle\hat{\mathscr{X}}\cup\{g,h\}\rangle$, for all $\alpha\in\tilde{\mathscr{X}}$ such that $\rank(\alpha)={\lfloor\frac{n}{2}\rfloor}+1$. 

Next, let $2\leqslant k\leqslant {\lfloor\frac{n}{2}\rfloor}+1$ and 
suppose, by induction hyphothesis, that $\gamma\in \langle\hat{\mathscr{X}}\cup\{g,h\}\rangle$, 
for all $\gamma\in\tilde{\mathscr{X}}$ such that $\rank(\gamma)\geqslant k$. 

Let $\alpha\in\tilde{\mathscr{X}}$ be such that $\rank(\alpha)=k-1$. Obviously, if $\alpha\in\hat{\mathscr{X}}$ then $\alpha\in \langle\hat{\mathscr{X}}\cup\{g,h\}\rangle$. 
So, suppose that $\alpha\in\tilde{\mathscr{X}}\setminus\hat{\mathscr{X}}$. 
Therefore, there exists $\beta\in\hat{\mathscr{X}}$ such that $\beta\preccurlyeq\alpha$ and $\rank(\alpha)<\rank(\beta)\leqslant {\lfloor\frac{n}{2}\rfloor}+1$. 
Let $\ell=\rank(\beta)$. 

Let $\alpha'\in[\alpha]_\sim$ and $\beta'\in[\beta]_\sim$ be such that $\ker(\beta')\subseteq\ker(\alpha')$. 
Take $\sigma_1,\sigma_2,\xi_1,\xi_2\in\mathscr{D}_{2n}$ such that $\alpha=\sigma_1\alpha'\sigma_2$ and $\beta'=\xi_1\beta\xi_2$. 

Now, by (\ref{trans}), we get $\im(\alpha' g^r)=\{1,2,\ldots,k-1\}$ and $\im(\beta' g^s)=\{1,2,\dots,\ell\}$, for some $0\leqslant r,s\leqslant n-1$. 
Moreover, $\ker(\beta' g^s)=\ker(\beta')\subseteq\ker(\alpha')=\ker(\alpha' g^r)$. 

If $k<{\lfloor\frac{n}{2}\rfloor}+1$ or $\alpha' g^r\in\End(C_n)$ then, by Lemma \ref{fund}, 
there exists $\gamma\in\wEnd(C_n)$ such that $\rank(\gamma)=k$ and $\alpha' g^r=\beta' g^s\gamma$. 
As $\tilde{\mathscr{X}}\cup\{g,h\}$ generates $\wEnd(C_n)$, $\gamma$ is a product of elements of $\tilde{\mathscr{X}}\cup\{g,h\}$, 
all of them with rank greater than or equal to $k$, and so, by induction hypothesis, $\gamma\in \langle\hat{\mathscr{X}}\cup\{g,h\}\rangle$. 
Thus 
$$
\alpha=\sigma_1\alpha'\sigma_2=\sigma_1\beta' g^s\gamma g^{n-r}\sigma_2 = \sigma_1\xi_1\beta\xi_2 g^s\gamma g^{n-r}\sigma_2 
\in \langle\hat{\mathscr{X}}\cup\{g,h\}\rangle.
$$
 
If $k={\lfloor\frac{n}{2}\rfloor}+1$ and $\alpha' g^r\not\in\End(C_n)$ then, by Lemma \ref{fund}, 
there exist $\lambda\in \hat{\mathscr{X}}$ and $\eta,\zeta\in \mathscr{D}_{2n}$ such that 
$\alpha' g^r=\beta' g^s\eta\lambda\zeta$. Thus 
$$
\alpha=\sigma_1\alpha'\sigma_2=\sigma_1\beta' g^s \eta\lambda\zeta g^{n-r}\sigma_2 = \sigma_1\xi_1\beta\xi_2 g^s \eta\lambda\zeta g^{n-r}\sigma_2
\in \langle\hat{\mathscr{X}}\cup\{g,h\}\rangle.
$$

Therefore, we proved that $\wEnd(C_n)=\langle\hat{\mathscr{X}}\cup\{g,h\}\rangle$. 

\medskip 

\noindent\textsc{step 2}.  
By replacing $\mathscr{X}$ by $\mathscr{Y}$, $\hat{\mathscr{X}}$ by $\hat{\mathscr{Y}}$, 
$\tilde{\mathscr{X}}$ by $\tilde{\mathscr{Y}}$ and $\wEnd(C_n)$ by $\End(C_n)$ in \textsc{step 1}, in a similar way 
(with the exception of the final part where it suffices to consider $\alpha' g^r\in\End(C_n)$), 
we show that $\hat{\mathscr{Y}}\cup\{g,h\}$ is a generating set of $\End(C_n)$. 

\medskip

\noindent\textsc{step 3}.  
Next, we prove that $\hat{\mathscr{X}}\cup\{g,h\}$ has minimum size among the generating sets of $\wEnd(C_n)$. 

Let $\mathscr{G}$ be a subset of $\wEnd(C_n)\setminus\mathscr{D}_{2n}$ such that $\mathscr{G}\cup\{g,h\}$ generates $\wEnd(C_n)$. 

\smallskip 

We begin by showing that, for all $\lambda\in\hat{\mathscr{X}}$, there exists $\alpha\in\mathscr{G}$ such that $\lambda\sim\alpha$. 

Take $\lambda\in\hat{\mathscr{X}}$. Since $\lambda$ is not a permutation and $\mathscr{G}\cup\{g,h\}$ generates $\wEnd(C_n)$, 
there exist $\sigma\in\mathscr{D}_{2n}$, $\alpha\in\mathscr{G}$ and $\beta_1\in\wEnd(C_n)$ such that $\lambda=\sigma\alpha\beta_1$. 
Hence, $\sigma^{-1}\lambda=\alpha\beta$ and so $\ker(\alpha)\subseteq\ker(\sigma^{-1}\lambda)$. 

Since $\alpha$ is not a permutation and $\hat{\mathscr{X}}\cup\{g,h\}$ generates $\wEnd(C_n)$, 
there exist $\xi\in\mathscr{D}_{2n}$, $\gamma\in\hat{\mathscr{X}}$ and $\beta_2\in\wEnd(C_n)$ such that $\alpha=\xi\gamma\beta_2$. 
Hence, $\xi^{-1}\alpha=\gamma\beta_2$ and so $\ker(\gamma)\subseteq\ker(\xi^{-1}\alpha)\subseteq\ker(\xi^{-1}\sigma^{-1}\lambda)$. 
As $\xi^{-1}\sigma^{-1}\lambda\sim\lambda$, we have $\gamma\preccurlyeq\lambda$. Then, since $\lambda,\gamma\in\hat{\mathscr{X}}$, 
by construction of $\hat{\mathscr{X}}$ (Algorithm \ref{algo}), we must have $\rank(\gamma)=\rank(\lambda)$ and so 
$\ker(\gamma)=\ker(\xi^{-1}\alpha)=\ker(\xi^{-1}\sigma^{-1}\lambda)$, from it follows also that $\ker(\alpha)=\ker(\sigma^{-1}\lambda)$. 
Thus, 
by Proposition \ref{kers}, we have 
$\alpha=\sigma^{-1}\lambda\tau$, for some $\tau\in\mathscr{D}_{2n}$, and so, in particular, $\lambda\sim\alpha$. 

\smallskip 

Now, for each $\lambda\in\hat{\mathscr{X}}$, choose $\alpha_\lambda\in\mathscr{G}$ such that $\lambda\sim\alpha_\lambda$. 
Let $\varphi:\hat{\mathscr{X}}\longrightarrow\mathscr{G}$ be the mapping defined by $\lambda\varphi=\alpha_\lambda$, for all $\lambda\in\hat{\mathscr{X}}$. 
Hence, $\varphi$ is an injective mapping. In fact, let $\lambda_1,\lambda_2\in\hat{\mathscr{X}}$ be such that $\lambda_1\varphi=\lambda_2\varphi$. 
Then $\lambda_1\sim\alpha_{\lambda_1}=\alpha_{\lambda_2}\sim\lambda_2$ and, since $\hat{\mathscr{X}}\subseteq\tilde{\mathscr{X}}$, 
it follows that $\lambda_1=\lambda_2$. Therefore, $|\hat{\mathscr{X}}|\leqslant |\mathscr{G}|$ and so $|\hat{\mathscr{X}}\cup\{g,h\}|\leqslant |\mathscr{G}\cup\{g,h\}|$,
as required. 

\medskip 

\noindent\textsc{step 4}. 
By replacing $\mathscr{X}$ by $\mathscr{Y}$, $\hat{\mathscr{X}}$ by $\hat{\mathscr{Y}}$, 
$\tilde{\mathscr{X}}$ by $\tilde{\mathscr{Y}}$ and $\wEnd(C_n)$ by $\End(C_n)$ in \textsc{step 3}, in a similar way,  
we show that $\hat{\mathscr{Y}}\cup\{g,h\}$ has minimum size among the generating sets of $\End(C_n)$, which completes the proof. 
\end{proof}

Using the GAP \cite{GAP4} to perform the calculations (in particular, by running an implementation of Algorithm \ref{algo} in GAP), we have: 
$$
\begin{tabular}{|c|c|c|c|c|c|c|c|c|c|c|} \hline 
$n$ & \multicolumn{2}{|c|}{$\Aut(C_n)$} & \multicolumn{2}{|c|}{$\sEnd(C_n)$} &\multicolumn{2}{|c|}{$\End(C_n)$} &\multicolumn{2}{|c|}{$\swEnd(C_n)$} &\multicolumn{2}{|c|}{$\wEnd(C_n)$} \\ \cline{2-11}
 & size & rank & size & rank & size & rank & size & rank & size & rank\\ \hline 
$3$ & $6$ & $2$ & $6$ & $2$ & $6$ & $2$ & $27$ & $3$ & $27$ & $3$ \\\hline 
$4$ & $8$ & $2$ & $32$ & $3$ & $32$ & $3$ & $36$ & $4$ & $84$ & $4$ \\\hline 
$5$ & $10$ & $2$ & $10$ & $2$ & $10$ & $2$ & $15$ & $3$ & $265$ & $4$ \\\hline 
$6$ & $12$ & $2$ & $12$ & $2$ & $132$ & $3$ & $18$ & $3$ & $858$ & $6$ \\\hline 
$7$ & $14$ & $2$ & $14$ & $2$ & $14$ & $2$ & $21$ & $3$ & $2765$ & $7$ \\\hline 
$8$ & $16$ & $2$ & $16$ & $2$ & $576$ & $4$ & $24$ & $3$ & $8872$ & $13$ \\\hline 
$9$ & $18$ & $2$ & $18$ & $2$ & $18$ & $2$ & $27$ & $3$ & $28269$ & $20$ \\\hline 
$10$ & $20$ & $2$ & $20$ & $2$ & $2540$ & $5$ & $30$ & $3$ & $89550$ & $50$ \\\hline 
$11$ & $22$ & $2$ & $22$ & $2$ & $22$ & $2$ & $33$ & $3$ & $282205$ & $105$ \\\hline 
$12$ & $24$ & $2$ & $24$ & $2$ & $11112$ & $10$ & $36$ & $3$ & $885492$ & $272$ \\\hline 
\end{tabular}
$$



\bigskip 

{\small \sf  

\noindent{\sc Ilinka Dimitrova},
Department of Mathematics,
Faculty of Mathematics and Natural Science,
South-West University "Neofit Rilski",
2700 Blagoevgrad,
Bulgaria;
e-mail: ilinka\_dimitrova@swu.bg.

\medskip

\noindent{\sc V\'\i tor H. Fernandes},
Center for Mathematics and Applications (NovaMath)
and Department of Mathematics, 
Faculdade de Ci\^encias e Tecnologia,
Universidade Nova de Lisboa,
Monte da Caparica,
2829-516 Caparica,
Portugal;
e-mail: vhf@fct.unl.pt.

\medskip

\noindent{\sc J\"{o}rg Koppitz},
Institute of Mathematics and Informatics,
Bulgarian Academy of Sciences,
1113 Sofia,
Bulgaria;
e-mail: koppitz@math.bas.bg.

\medskip

\noindent{\sc Teresa M. Quinteiro},
Instituto Superior de Engenharia de Lisboa,
1950-062 Lisboa,
Portugal.
Also:
Center for Mathematics and Applications (NovaMath),
Faculdade de Ci\^encias e Tecnologia,
Universidade Nova de Lisboa,
Monte da Caparica,
2829-516 Caparica,
Portugal;
e-mail: tmelo@adm.isel.pt.
} 

\end{document}